\documentclass{aptpub}

\usepackage[english]{babel}
\usepackage[utf8x]{inputenc}
\usepackage{subcaption,commath,graphicx,color}

\authornames{D.~RUDOLF AND A.Q.~WANG} 
\shorttitle{Perturbation theory for QSDs} 

\newcommand{\R}{\mathbb R}
\newcommand{\Rd}{\mathbb R^d}

\renewcommand{\P}{\mathbb P}

\newcommand{\Lk}{L^\kappa}

\newcommand{\taud}{\tau_\partial}
\newcommand{\lkap}{\lambda_0^\kappa}

\newcommand{\Dom}{\mathcal D}
\newcommand{\lspan}{\mathrm{span}}
\renewcommand{\H}{\mathcal H}

\begin{document}

\title{Perturbation theory for killed Markov processes and quasi-stationary distributions} 

\authorone[Universit\"at Passau]{Daniel Rudolf} 

\addressone{Faculty of Computer Science and Mathematics,
Universit\"at Passau, Innstrasse 33, 94032 Passau, Germany} 
\emailone{daniel.rudolf@uni-passau.de} 

\authortwo[University of Warwick]{Andi Q. Wang}
\addresstwo{Department of Statistics, University of Warwick, Coventry, CV4 7AL, UK}
\emailtwo{andi.wang@warwick.ac.uk}

\begin{abstract}
 Motivated by recent developments of quasi-stationary Monte Carlo methods, we investigate the stability of quasi-stationary distributions of killed Markov processes under perturbations of the generator. 
    We first consider a general bounded self-adjoint perturbation operator, and after that, study a particular unbounded perturbation corresponding to truncation of the killing rate.
    In both scenarios, we quantify the difference between eigenfunctions of the smallest eigenvalue of the perturbed and unperturbed generators in a Hilbert space norm. As a consequence, $\mathcal{L}^1$ norm estimates of the difference of the resulting quasi-stationary distributions in terms of the perturbation are provided.
\end{abstract}



\section{Introduction}

Perturbation theory concerns the investigation of
the stability of a given system under small changes. In the context of Markov chains, one typically studies the impact of perturbing the transition mechanism, namely the transition kernel or matrix, on the behavior of the corresponding chain. For example, one can study the resultant effect on the limiting distribution and approximation properties, see e.g.
\cite{kartashov2019strong,roberts1998convergence,mitrophanov2003stability,mitrophanov_2005,rudolf2018} and the references therein.

In this work, we extend the study of perturbation theory to the setting of \textit{killed} Markov processes which possess \textit{quasi-stationary distributions}. This is particularly motivated by recent developments in Markov chain Monte Carlo (MCMC), where killed diffusions have been used to construct scalable algorithms to perform Bayesian inference for large data sets, see \cite{Pollock2020, Kumar2019}. Convergence properties of this approach, known as \textit{quasi-stationary Monte Carlo} (QSMC), have been investigated in \cite{Wang2019TheoQSMC}.

Given that such methods have been applied to perform Bayesian inference, it is important to study their stability. We ask the following question:
Can the quasi-stationary distribution be adversely affected by small perturbations of the system, which could arise, for instance, due to data corruption, or missing data?

We start by modeling the perturbation as a bounded self-adjoint operator and quantify the difference in a general setting between the eigenfunctions of the smallest eigenvalues with respect to the perturbed and unperturbed generators, see Theorem~\ref{thm:bdd_case}.
As a consequence, one can develop $\mathcal{L}^1$-bounds on the difference between the corresponding quasi-stationary distributions, which is illustrated in the QSMC framework for logistic regression, see Proposition~\ref{prop:logistic_bd}.




In addition, we also consider a specific unbounded perturbation, which corresponds to truncating a divergent killing rate.
In practical implementations of QSMC methods, such a truncation can significantly reduce the complexity of the implementation, since simulating arrivals from a Poisson process with a bounded intensity is straightforward using Poisson thinning, whereas unbounded intensities are much more intricate. 
We give an explicit bound on the $\mathcal{L}^2$-norm, depending on the truncation level, of the difference of the eigenfunctions of the smallest eigenvalue of the generators with and without truncation, see Theorem~\ref{thm:unbdd_case}. We apply this result to an Ornstein-Uhlenbeck diffusion with killing and obtain an estimate of the difference between the perturbed quasi-stationary distribution and a Gaussian target distribution in terms of an $\mathcal{L}^1$-norm, see Propositions~\ref{prop:unbdd_L1}, \ref{prop:multivar_trunc}. In this setting, we find that truncation is robust with respect to the dimension; in any dimension $d\ge 1$, the error decays exponentially as a function of the truncation level.


\subsection{Related results and literature review}

Quasi-stationary distributions have been, and continue to be, actively studied within the probability theory literature; see for instance the book of \cite{Collet2013} or the review paper \cite{Meleard2012}. More recently, the construction of schemes with a desired quasi-stationary distribution has been investigated in the context of computational statistics, see \cite{Pollock2020, Kumar2019, Wang2019TheoQSMC}. Quasi-stationary distributions have also been utilized recently as a tool to analyse other algorithms; see for instance \cite{Lelievre2015, Baudel2020} for analyses of Monte Carlo algorithms in the context of molecular dynamics. 

Perturbation theory in the context of Markov chains (without killing) is richly studied in the literature; see for instance the references \cite{kartashov2019strong,roberts1998convergence,mitrophanov2003stability, mitrophanov_2005,johndrow2017error, rudolf2018,Medina-Aguayo2018,hosseini2018convergence, Negrea2021,fuhrmann2021wasserstein}. 
However, there are key differences between the classical Markov chain setting and our quasi-stationary setting. For a Markov chain with killing, we typically study the conditional distribution after $n$ steps \textit{conditional on survival}. This distribution often converges to a quasi-stationary distribution, whereas in the classical Markov chain setting there is no such conditioning. This makes the approximate discrete-time Markov chain with killing setting more demanding, and we present a result in the subsequent Section~\ref{subsec:discrete-time}. 

However, following our motivation of QSMC methods \cite{Pollock2020}, the main focus of this work will be on continuous-time diffusion processes with killing and their quasi-stationary distributions.
This is a departure from the traditional Markov chain perturbation theory literature, as cited above, and requires new tools.
Thus we apply functional analytic tools, building on techniques developed for the study of perturbations of linear operators, a rich and mature field \cite{Reed1978, Baumgartel1985,Kato1995}. Several of our results will thus consist of translating these abstract results appropriately into the setting of killed diffusions. 

\subsection{Result for discrete-time Markov chains}\label{subsec:discrete-time}
We now give a result for discrete-time Markov chains with killing; in our discussion \cite{Rudolf2020}
this has already been reported without proof.
We repeat the specific setting and the statement:
Let $(S_0,\mathcal{S}_0)$ be a measurable space and $\partial\not\in S_0$ an additional element, which serves as ``trap'' of the killed Markov chain. Moreover, $Q\colon S_0\times \mathcal{S}_0 \to [0,1]$ is a transition kernel and $\kappa \colon S_0 \to [0,1]$ is used to describe a state-dependent killing probability. For a probability space $(\Omega,\mathcal{F},\mathbb{P})$, let $(X_n)_{n\in\mathbb{N}_0}$ be a  sequence of random variables $X_n\colon \Omega \to S_0\cup \{\partial \}$ determined by the following transition mechanism, depending on $Q$ and $\kappa$ with $X_0:=x_0\in S_0$. Given $X_n=x$ with $x\in S_0 \cup \{\partial \}$, the distribution of $X_{n+1}$ is specified by two steps:
\begin{enumerate}
    \item If $x=\partial$, then return $X_{n+1}:= \partial$;
    \item \label{enu:step2} Otherwise, draw a uniformly distributed sample $u\sim \textrm{Unif}[0,1]$. Independently, draw a sample $y\sim Q(x,\cdot)$. Then, return
    \[
        X_{n+1}:=
        \begin{cases}
            \partial & \textrm{if } u\leq \kappa(x),\\
            y & \textrm{if } u>\kappa(x).
        \end{cases}
    \]
\end{enumerate}
Thus, $(X_n)_{n\in\mathbb{N}_0}$ is a Markov chain with killing, determined by $Q$, $\kappa$ and initialization $x_0$. Now suppose we only have access to $\kappa$ through an approximation $\widetilde\kappa\colon S_0 \to [0,1]$ with $\Vert \kappa-\widetilde{\kappa} \Vert_\infty$ small. Setting $\widetilde{X}_0 := x_0$, let $(\widetilde{X}_n)_{n\in\mathbb{N}_0}$ be another sequence of random variables on $(\Omega,\mathcal{F},\mathbb{P})$ which is constructed identically to $(X_n)_{n\in\mathbb{N}_0}$, except that
$\widetilde{\kappa}$ is used in step~\ref{enu:step2} above instead of $\kappa$. 
For the convenience of the reader, we write $\mathbb{P}_{x_0}$ instead of $\mathbb{P}$ to indicate the initialization at $x_0$.
Now, under suitable regularity conditions we can quantify the difference of the conditional distributions of interest, derived from the Markov chains with killing rates $\kappa$ and $\widetilde{\kappa}$.
\begin{proposition} \label{prop: discrete_time_perturbation}
Assume that there exist $\alpha,\widetilde{\alpha}\in (0,1)$ and $c_\ell,\widetilde{c}_\ell,c_u,\widetilde{c}_u\colon S_0 \to (0,\infty)$, such that for any $n\in\mathbb{N}$ we have
\begin{align}
\label{al: exp_decay_discrete_time1}
\alpha^n c_\ell(x_0) & \leq \mathbb{P}_{x_0}(X_n\in S_0) \leq \alpha^n c_u(x_0),\\
\label{al: exp_decay_discrete_time2}
    \widetilde\alpha^n \widetilde c_\ell(x_0) & \leq \mathbb{P}_{x_0}(\widetilde{X}_n\in S_0) \leq \widetilde\alpha^n \widetilde c_u(x_0).
\end{align}
Then
	\begin{align*}
  &	\left \Vert \mathbb{P}_{x_0}(X_n\in \cdot\mid X_n\in S_0) - \mathbb{P}_{x_0}(\widetilde{X}_n\in \cdot\mid \widetilde{X}_n\in S_0)\right \Vert_{\text{\rm tv}} 
	\leq   \Vert \kappa - \widetilde{\kappa} \Vert_{\infty} K(x_0) \min\{n,\vert \alpha-\widetilde{\alpha} \vert^{-1}\}, \end{align*}
	where 
	\[
		K(x_0) := 2 \sup_{z\in S_0} \int_{S_0}  \frac{\widetilde{c}_u(y) c_u(x_0)}{\min\{c_\ell(x_0),\widetilde{c}_\ell(x_0)\}}\, Q(z,{\rm d} y),
	\]
	and $\Vert \cdot \Vert_{\text{\rm tv}}$ denotes the total variation distance.
\end{proposition}
Note that for irreducible chains on finite state spaces the conditions \eqref{al: exp_decay_discrete_time1} and \eqref{al: exp_decay_discrete_time2} are generally mild assumptions; these can be straightforwardly seen to hold for some $\alpha, \widetilde \alpha\in (0,1)$ by the Perron--Frobenius theorem, \cite{Seneta2006}.

At first glance (for finite $K(x_0)$), this would appear to be a particularly good estimate, however, in the regime when $\Vert \kappa - \widetilde{\kappa} \Vert_{\infty}$
is small, it is probable that $\vert \alpha-\widetilde{\alpha} \vert$ is also small. Regardless, given $n$ and if $\Vert \kappa - \widetilde{\kappa} \Vert_{\infty}$ is negligible compared to $n^{-1}$ --- say $\Vert \kappa - \widetilde{\kappa} \Vert_{\infty}\leq n^{-2}$ --- then the right hand-side of the stated estimate is of order $n^{-1}$.

The proof of Proposition~\ref{prop: discrete_time_perturbation} is provided in Appendix~\ref{sec:Discrete_time}. A key aspect in the reasoning is handling the different normalizing constants in the conditional probabilities. 
It is not clear how this
approach could be extended to the continuous-time killed Markov process setting, on which the QSMC setting relies. Therefore we follow here a different route based on Hilbert space techniques and spectral properties.


\subsection{Background to quasi-stationary Monte Carlo}
\label{sec:background}

We provide some background to the framework and theoretical properties used in the development of QSMC schemes, largely following \cite{Wang2019TheoQSMC}.

For fixed dimension $d\in\mathbb{N}$, let $\nabla$ be the gradient operator, i.e., the $d$-dimensional vector with components $\nabla_i:=\partial /\partial x_i$ for $i=1,\dots,d$. Let $X=(X_t)_{t\geq0}$ be the $d$-dimensional reversible diffusion, which is defined as the (weak) solution of the stochastic differential equation (SDE)
\begin{equation}
    \dif X_t = \nabla A (X_t) \dif t + \dif W_t.
    \label{eq:SDE}
\end{equation}
Here $W=(W_t)_{t\geq 0}$ denotes a standard $d$-dimensional Brownian motion
and 
$A:\mathbb R^d \to \mathbb R$ is an infinitely differentiable function such that the SDE \eqref{eq:SDE} has a unique non-explosive weak solution. To be more precise: 
Let $C$ be the space of all continuous functions mapping from the non-negative real numbers into $\mathbb{R}^d$.
For $\omega \in C$ and each $t\geq 0 $ define $X_t \colon C \to \mathbb{R}^d$ to be the coordinate map $X_t(\omega)=\omega(t)$ and let $\mathcal{C}:= \sigma(\{X_t\colon t\geq 0 \})$. For $X_0=x$, with $x\in\mathbb{R}^d$, let $\widetilde\P_x$ be the probability measure on $(C,\mathcal{C})$, such that $X=(X_t)_{t\geq0}$ is the weak solution of \eqref{eq:SDE} under $\widetilde{\mathbb{P}}_x$. For arbitrary initial distribution $\mu$ on $\Rd$, that is, $X_0\sim \mu$, we have $\widetilde\P_\mu(\cdot) = \int_{\Rd} \widetilde\P_x(\cdot) \mu(\dif x)$.
We then augment the probability space with an additional independent unit exponential random variable $\xi$, denoting the resulting probability space by $(\Omega,\mathcal{F},\mathbb{P}_x)$ (or $(\Omega,\mathcal{F},\mathbb{P}_\mu)$ respectively).
Fix now a measurable non-negative function $\kappa: \Rd \to [0,\infty)$, the \textit{killing rate}, and define the killing time,
\begin{equation}
    \taud := \inf \left\{ t\ge 0: \int_0^t \kappa(X_s) \dif s \ge \xi \right\}.
    \label{eq:taud}
\end{equation}
 By convention, set $\inf \emptyset =\infty$.

Now let $\pi$ be the probability measure of interest on $\mathbb{R}^d$ given through a positive, infinitely-differentiable and integrable Lebesgue density. By an abuse of notation we also denote this density by $\pi$. The idea is to choose the function $\kappa$ in such a way that $\pi$ is 
the unique \textit{quasi-stationary distribution of $X$ killed at rate $\kappa$}: for each $t\ge 0$, we have that
\begin{equation}
\label{eq: quasi_stat}
    \P_\pi (X_t \in \cdot \mid \taud>t)=\pi.
\end{equation}
In \cite{Wang2019TheoQSMC} it has been shown that \eqref{eq: quasi_stat} is satisfied, under mild regularity conditions, if 
\begin{equation}
    \kappa = \frac{1}{2}\left( \frac{\Delta \pi}{\pi} -\frac{2\nabla A \cdot \nabla \pi}{\pi} -2\Delta A \right)+K,
    \label{eq:kappa_QSMC}
\end{equation}
with $K>0$ being chosen such that $\kappa\geq0$, and with $\Delta := \sum_{i=1}^d \frac{\partial^2}{\partial^2 x_i}$ being the Laplacian operator. In this reference it is shown that $\pi$ is the unique quasi-stationary distribution of $X$ and that the convergence of the distribution of $X_t$ given $\tau_\partial >t$ to $\pi$ can take place geometrically.



\subsection{Outline}

In Section~\ref{sec:bdd_pert}, we consider the case of general bounded self-adjoint perturbations for self-adjoint generators. We give a general perturbation result in Theorem~\ref{thm:bdd_case}, and then apply this to logistic regression in Section~\ref{subsec:logistic}, which is the main current application of QSMC methods; see \cite{Pollock2020, Kumar2019}.

In Section~\ref{sec:truncation}, we turn to study one particular \textit{unbounded} perturbation, which corresponds to truncating a divergent killing rate. We first give a general perturbation result in Theorem~\ref{thm:unbdd_case}, and then apply this to an  Ornstein--Uhlenbeck setting where explicit bounds on the incurred bias can be derived.

Finally, in the Appendix we present some deferred proofs.

\section{Bounded perturbations}
\label{sec:bdd_pert}

In this section we provide a self-contained proof of a norm-estimate of the difference between the eigenfunctions of two (potentially unbounded) self-adjoint operators, where one is considered to be a bounded perturbation of the other.

In the context of quasi-stationary Monte Carlo, such bounds are reassuring for practitioners wishing to implement QSMC methods, since our bound will say that small self-adjoint perturbations of the generator will not substantially affect the resulting quasi-stationary distribution.
Such self-adjoint perturbations can arise in practical situations when the target density has been perturbed slightly, since the target density typically only enters the QSMC algorithm via the killing rate (which is a multiplication operator, hence self-adjoint). For example, in a logistic regression scenario, such a perturbation could correspond to \textit{corrupted} or \textit{missing} data; see Section~\ref{subsec:logistic}.

\subsection{Assumptions and first consequences}
Let $\mathcal{H}$ be a real separable Hilbert space with inner-product $\langle\cdot,\cdot\rangle$  and denote by $\mathsf S(\H):=\{f\in\H:\|f\|=1\}$ the corresponding unit sphere. For a linear subspace $\mathbb V$, we write $\mathbb V^\perp$ for its orthogonal complement and define the linear operator $P_{\mathbb{V}}\colon \mathcal{H} \to \mathbb{V}$ as the orthogonal projection from $\mathcal{H}$ to $\mathbb{V}$.

Suppose we have a densely-defined self-adjoint positive semi-definite linear operator $L \colon \mathcal{H} \to \mathcal{H}$ with domain $\Dom(L) \subseteq \mathcal{H}$; we will simply write $(L, \Dom(L))$. Recall that the positive semi-definiteness implies that $\langle Lf,f\rangle \ge 0$ for all $f\in \Dom(L)$. Then the spectrum of $L$, denoted by $\sigma(L)$, is nonempty and non-negative: $\sigma(L) \subset [0,\infty)$; \textit{c.f.} \cite[Proposition 5.12]{Hislop1996}.
We define
\begin{equation*}
    \begin{split}
        \lambda_0 &:= \inf \sigma(L), \\
        \lambda_1 &:= \inf\left \{\sigma(L)\backslash \{\lambda_0\} \right \},
    \end{split}
\end{equation*}
which are crucial quantities that appear in the following assumption.
\begin{assumption}
	For the operator $L$ we have that $\nu := \lambda_1-\lambda_0 >0$, that is, $L$ possesses a spectral gap. Additionally, we assume that the eigenspace corresponding to $\lambda_0$ is $1$-dimensional.
    \label{assm:spec_gap}
\end{assumption}
Note that 
under the assumption $\nu>0$, $\lambda_0$ is indeed an eigenvalue of $L$ (see, e.g. \cite[Theorem 4.1]{Kulkarni2008}), which we assume to have a 1-dimensional eigenspace. Therefore, we have $L\varphi = \lambda_0 \varphi$, where $\varphi\in \mathsf S(\H)$ denotes an eigenfunction of $L$ with eigenvalue $\lambda_0$.
By \cite[Theorem XIII.1]{Reed1978} we in fact have 
the following variational representations of $\lambda_0$ and $\lambda_1$:
\begin{equation}
    \begin{split}
        \lambda_0 &= \inf_{f\in \mathsf S(\H)\cap \Dom(L)} \langle f,Lf\rangle,\\
        \lambda_1 &= 
       \sup_{\mathbb V \in \mathcal V_1} 
\inf_{f\in \mathbb{V}^\perp \cap \mathsf S(\H)\cap \Dom(L)} \langle f, Lf\rangle,
    \end{split}
    \label{eq:var_eigenv}
\end{equation}
where $\mathcal V_1$ denotes the set of $1$-dimensional linear subspaces of $\mathcal{H}$.
Note that $\lambda_1$ is not necessarily an eigenvalue of $L$.

Fix a bounded self-adjoint linear operator $H\colon \mathcal{H} \to \mathcal{H}$, which serves as a ``small'' perturbation. With this we set
$$\widehat L := L+H.$$ Note that we have $\Dom(\widehat L) = \Dom(L)$,  
self-adjointness of $\widehat{L}$ as well as
$\langle \widehat L f, f\rangle \geq -\Vert H \Vert$ for any $f\in \Dom(L)$, with $\Vert H \Vert$ being the operator norm of $H$.
Define
\begin{equation}
    \begin{split}
        \widehat \lambda_0 &:=  \inf_{f\in \mathsf S(\H)\cap \Dom(L)} \langle f,\widehat Lf\rangle,\\
        \widehat\lambda_1 &:= \sup_{\mathbb V \in \mathcal V_1} \inf_{f\in \mathbb V^\perp \cap \mathsf S(\H)\cap \Dom(L)} \langle f, \widehat Lf\rangle.
    \end{split}
    \label{eq:def_hatlambda}
\end{equation}
The following lemma is a classical result; see e.g. \cite[Section~XIII.15]{Reed1978}. For completeness we give a proof in Appendix~\ref{app:weyl}.
\begin{lemma}[Weyl's lemma]
	\label{lemma:wey}
    For $j=0,1$, we have 
    \begin{equation}
        |\lambda_j - \widehat \lambda_j|\le \|H\|.
        \label{eq:eval_bd}
    \end{equation}
\end{lemma}

In the next assumption we formalize the notion that $H$ is a ``small'' perturbation.
\begin{assumption}
    Suppose that the operator norm of $ H$ satisfies $\|H\|<\nu/2$, where $\nu = \lambda_1-\lambda_0$ is the spectral gap; see Assumption~\ref{assm:spec_gap}.
    \label{assm:H_small}
\end{assumption}
Under the formulated assumptions we have two important consequences which follow from fairly standard functional analytic techniques; see  Appendix~\ref{app:lambda0} for a full proof.
\begin{lemma} \label{lem:conseq_Weyl_Ass2}
	Under Assumption~\ref{assm:spec_gap} and~\ref{assm:H_small} we have that
	$\widehat{\lambda}_0 <  \widehat{\lambda}_1$, and that $\widehat \lambda_0$ is an eigenvalue of $\widehat{L}$. Furthermore,  the perturbed operator $\widehat L$ possesses a unique  eigenfunction, up to a sign, $\widehat \varphi \in \mathsf S(\H)$ w.r.t. eigenvalue $\widehat{\lambda}_0$, that is, 
    \begin{equation*}
        \widehat L\widehat \varphi
        = \widehat \lambda_0\widehat \varphi.
    \end{equation*}
	Without loss of generality, we may assume that $\rho := \langle \varphi, \widehat \varphi\rangle \ge 0$. 
\end{lemma}


\begin{remark} \label{rem: conseq_eigenvalue}
    A consequence of the previous results is that
    \begin{equation}
        \widehat \lambda_0 = \langle \widehat \varphi ,\widehat L \widehat \varphi \rangle 
        = \inf_{f\in \mathsf S(\H)\cap \Dom(\widehat L)} \langle f,\widehat Lf\rangle,
        \label{eq:var_hatv0}
    \end{equation}
    and the supremum in the definition of $\widehat \lambda_1$ in equation \eqref{eq:def_hatlambda} is attained by $\mathbb V = \lspan\{\widehat{\varphi}\}$.
\end{remark}

\subsection{Main estimate}

The goal now is to estimate 
$\|\varphi-\widehat \varphi\|$. 
Now we state and prove the main estimate.
\begin{theorem}
    Under Assumption~\ref{assm:spec_gap} and~\ref{assm:H_small} we have
    \begin{equation} \label{eq_thm_bnd_case}
        \|\varphi-\widehat \varphi\| 
        \leq \frac{\| H\varphi \|}{\nu-2\|H\|}.
    \end{equation}
    \label{thm:bdd_case}
\end{theorem}
\begin{proof}
Recall that $\rho := \langle \widehat \varphi, \varphi\rangle \ge 0$ and with $\mathbb{G}:=\lspan\{ \varphi\}$ define
	\begin{equation*}
	h := \frac{P_{\mathbb{G}^\perp} \widehat{\varphi}}{\Vert P_{\mathbb{G}^\perp} \widehat{\varphi} \Vert}
	= (1-\rho^2)^{-1/2} (\widehat \varphi - \rho\, \varphi).
	\end{equation*}
	(If $\rho=1$, then $\widehat \varphi = \varphi$ and the result trivially holds true.) Note that $\|h\|=1$
	and $h\perp \varphi$ as well as $L h \perp \varphi$ by the self-adjointness of $L$. 
 Consequently, by using the representation of $h$, $L\varphi = \lambda_0 \varphi$ and $\widehat L=L+H$ we have that
 \begin{align*}
  \rho \lambda_0 \varphi + \rho H \varphi + \sqrt{1-\rho^2} \widehat L h =
  \rho \widehat L \varphi + \widehat L (\widehat \varphi -\rho \varphi)
  = 
    \widehat L\widehat\varphi
= \widehat{\lambda}_0 \widehat\varphi.
 \end{align*}
Taking the inner product with $h$, employing $\langle \widehat \varphi,h \rangle = \sqrt{1-\rho^2}$, $\varphi\perp h$ and rearranging yields
\begin{equation}
\label{eq: repres}
    -\rho \langle H \varphi, h\rangle = \sqrt{1-\rho^2} \left( \langle h,\widehat L h\rangle -\widehat\lambda_0 \right).
\end{equation}
Moreover, by $\widehat{L}=L+H$, $h\perp \varphi$ and Lemma~\ref{lemma:wey} we have 
\begin{align*}
   B & := \langle h,\widehat{L} h \rangle - \widehat\lambda_0
    = \langle h,L h \rangle + \langle h,H h \rangle - \widehat\lambda_0 
    \geq \lambda_1 + \langle h,H h \rangle - \widehat\lambda_0  \\
    &\geq \lambda_1 - \Vert H \Vert - \widehat\lambda_0
    \geq \lambda_1 - 2\|H\|-\lambda_0=\nu-2\|H\|,
\end{align*}
which is, due to Assumption~\ref{assm:H_small}, strictly positive. 
Squaring \eqref{eq: repres}, rearranging, and taking a square root, with $A:=|\langle H \varphi ,h\rangle|$ we obtain
\begin{equation*}
    \rho = \frac{B}{\sqrt{A^2 + B^2}}.
\end{equation*}
Then
\begin{equation*}
    1-\rho = \frac{\sqrt{A^2+B^2}-B}{\sqrt{A^2+B^2}} = \frac{A^2}{(\sqrt{A^2+B^2}+B)\sqrt{A^2+B^2}}\le\frac{A^2}{2B^2},
\end{equation*}
from which we have that $\sqrt{2(1-\rho)}\le A/B$.
The fact that we can represent $\widehat \varphi$ by
	the orthogonal decomposition
	\begin{equation*}
	\widehat \varphi
	= P_{\mathbb{G}} \widehat \varphi +P_{\mathbb{G}^\perp} \widehat \varphi
	=
	\rho\, \varphi + \sqrt{1-\rho^2 } h
	\end{equation*}
 gives, with $\rho = \langle \widehat \varphi, \varphi\rangle \ge 0$, that
$
\widehat \varphi - \varphi = (\rho-1)\varphi + \sqrt{1-\rho^2} h.
$
Consequently, 
\[
\Vert \widehat\varphi -\varphi \Vert = \sqrt{2(1-\rho)} \leq \frac{A}{B} \leq \frac{\Vert H\varphi \Vert}{\nu - 2\Vert H \Vert},
\]
which finishes the proof.

\end{proof}

\begin{remark} \label{rem:bdd_DK}
	A similar, and to some extent complementary, bound to \eqref{eq_thm_bnd_case} can be obtained via the sin~$\theta$ theorem of \cite{DavisKahan1970}, which is stated in Section~\ref{subsec:projection_notation}.
	Under Assumptions~\ref{assm:spec_gap} and~\ref{assm:H_small}, we obtain
	\begin{equation}
	\label{eq: DavisKahan_bnd}
	\|\varphi-\widehat \varphi\| \le \frac{2\sqrt{2}\|H\varphi\|}{\nu}.
	\end{equation}
	For the convenience of the reader we provide in Appendix~\ref{sec:proof_Davis_Kahan_bnd} arguments to derive \eqref{eq: DavisKahan_bnd}.
\end{remark}

\subsection{Additional auxiliary estimates}

We now give some additional estimates which are useful for our subsequent examples. 
Let $\pi,\widetilde \pi$ be probability density functions on $\Rd$. Fix a positive function $\gamma: \Rd \to (0,\infty)$, and let $\Gamma(\dif x)=\gamma(x)\dif x$ be the corresponding measure on $\Rd$. 
By $\mathcal L^2(\Rd,\Gamma)$ we denote the space of $\Gamma$-square-integrable functions $f\colon \mathbb{R}^d \to \mathbb{R}$ equipped with the norm 
$\Vert f \Vert_2 := \left( \int_{\mathbb{R}^d} \vert f(x) \vert^2 \,\Gamma(\dif x) \right)^{1/2}$. Now define the functions
\begin{equation*}
    \phi(x) := \frac{\pi(x)}{\gamma(x)},\qquad \widetilde \phi(x) := \frac{\widetilde \pi(x)}{\gamma(x)}, \qquad x\in\mathbb{R}^d.
\end{equation*}
We assume that $\phi,\widetilde \phi \in \mathcal L^2(\Gamma)$: we have that 
\begin{equation}
    \int_{\mathbb{R}^d} \frac{\pi(x)^2}{\gamma(x)}\dif x <\infty, \qquad \int_{\mathbb{R}^d} \frac{\widetilde \pi(x)^2}{\gamma(x)}\dif x<\infty.
    \label{eq:assm_L2}
\end{equation}
Under this condition we state the first auxiliary estimate.
\begin{lemma}\label{lemma:L1_to_L2}
    Assume that \eqref{eq:assm_L2} holds, and furthermore that $\Gamma$ is a finite measure, i.e., $\Lambda:= \Gamma(\mathbb{R}^d)^{1/2}<\infty$.
      Then
    \begin{equation*}
        \int_{\mathbb{R}^d}|\pi(x)-\widetilde \pi(x)|\dif x \le \Lambda \|\phi-\widetilde \phi\|_2.
    \end{equation*}
\end{lemma}
\begin{proof}
    We straightforwardly calculate
    \begin{equation*}
        \begin{split}
             \int_{\mathbb{R}^d}|\pi(x)-\widetilde \pi(x)|\dif x&=\int_{\mathbb{R}^d} |\phi-\widetilde \phi|\dif \Gamma \\
             &\le \left( \int_{\mathbb{R}^d} |\phi-\widetilde \phi|^{2}\dif \Gamma \right)^{1/2}\left (\int_{\mathbb{R}^d} \dif \Gamma \right )^{1/2} 
             = \Lambda \|\phi-\widetilde \phi\|_2,
        \end{split}
    \end{equation*}
    where we used Cauchy--Schwarz inequality.
\end{proof}
For the formulation of the next statement we need normalized functions $\phi$ and $\widetilde{\phi}$ w.r.t. the $\mathcal L^2(\Rd,\Gamma)$-norm, that is, we define 
\begin{equation*}
    \varphi(x) := \frac{\phi(x)}{Z},\qquad \widetilde \varphi(x) := \frac{\widetilde \phi(x)}{\widetilde Z},
\end{equation*}
with $Z:= \|\phi\|_2, \widetilde Z:= \|\widetilde \phi\|_2$. \begin{lemma}
    Suppose that 
    \eqref{eq:assm_L2} holds and $\Lambda:= \Gamma(\mathbb{R}^d)^{1/2}<\infty$. Additionally, let  
    \begin{equation}
        \|\varphi-\widetilde \varphi\|_2 \le \epsilon,
        \label{eq:assm_L2_bd}
    \end{equation}
    with $0<\epsilon<\Lambda^{-1} Z^{-1}$ be satisfied.
    Then, we have
        \begin{equation}
        \left \| \phi-\frac{Z}{\widetilde Z}\,\widetilde \phi\, \right \|_2 \le Z\epsilon,
        \label{eq:phi-zzphi}
    \end{equation}
    and
    \begin{equation}
        |Z-\widetilde Z|\le C \epsilon,
        \label{eq:norm_const_bd}
    \end{equation}
    with $C=Z^2 \Lambda/(1-\Lambda Z \epsilon)$.
    \label{lemma:norm_const_bd}
\end{lemma}
\begin{proof}
    Note that \eqref{eq:assm_L2_bd} immediately implies \eqref{eq:phi-zzphi}.
    Then, we have
    \begin{align*}
        \left |1-\frac{Z}{\widetilde Z}\right|&=\left | \int_{\mathbb{R}^d} \left(\phi-\frac{Z}{\widetilde Z}\,\widetilde \phi\right)\dif \Gamma \right|
        \le \int \left| \phi - \frac{Z}{\widetilde Z}\,\widetilde \phi  \,\right|\dif \Gamma\\
        &\le \left( \int_{\mathbb{R}^d}\left| \phi - \frac{Z}{\widetilde Z}\,\widetilde \phi  \,\right|^2\dif \Gamma  \right)^{1/2} \left( \int_{\mathbb{R}^d} \dif \Gamma\right)^{1/2}
        \le \Lambda Z \epsilon.
    \end{align*}
    We can then conclude \eqref{eq:norm_const_bd}, for instance by considering the behaviour of the function $t\mapsto h(t):=Z/t$. Namely, the former inequality implies
    $
        h(Z)-\Lambda Z\varepsilon \leq h(\widetilde{Z}) \leq  h(Z)+\Lambda Z\varepsilon.
    $
    Now, taking the inverse and performing suitable transformations yields the result.
\end{proof}
\begin{lemma}
    Assume that \eqref{eq:assm_L2} and \eqref{eq:assm_L2_bd} hold as well as that $\Lambda:= \Gamma(\mathbb{R}^d)^{1/2}<\infty$. Then, we have
    \begin{equation*}
        \|\phi-\widetilde \phi\|_2 \le \epsilon(Z+C),
    \end{equation*}
    where $C>0$ is defined as in Lemma~\ref{lemma:norm_const_bd}.
    \label{lemma:phi_2-bound}
\end{lemma}
\begin{proof}
    We calculate
    \begin{equation*}
        \begin{split}
            \|\phi-\widetilde \phi\|_2 &\le \left \|\phi - \frac{Z}{\widetilde Z}\,\widetilde \phi\,\right \|_2+\left \| \frac{Z}{\widetilde Z}\,\widetilde \phi - \widetilde \phi \,\right\|_2
            = \left \|\phi - \frac{Z}{\widetilde Z}\,\widetilde \phi\,\right \|_2+|Z-\widetilde Z| 
            \le Z\epsilon + C\epsilon, 
        \end{split}
    \end{equation*}
    where we have used the triangle inequality and Lemma~\ref{lemma:norm_const_bd}.
\end{proof}
Putting together the estimates of the former lemmas we obtain the following.
\begin{proposition}
    Assume that \eqref{eq:assm_L2}, \eqref{eq:assm_L2_bd} hold, and that $\Gamma$ is a finite measure, with $\Lambda := \Gamma(\mathbb{R}^d)^{1/2}<\infty$. Then, 
    \begin{equation*}
        \int_{\mathbb{R}^d} |\pi(x)-\widetilde \pi(x)|\dif x \le \Lambda(Z+C) \epsilon, 
    \end{equation*}
    where $C>0$ is defined in Lemma~\ref{lemma:norm_const_bd}.
    \label{prop:L1_pi_bound}
\end{proposition}

\subsection{Example: logistic regression}
\label{subsec:logistic}
We demonstrate how Theorem~\ref{thm:bdd_case}, in combination with Proposition~\ref{prop:L1_pi_bound}, can be applied to the setting of logistic regression, which is currently the most popular application of QSMC methods, see \cite{Pollock2020, Kumar2019}. In that scenario we justify that QSMC methods are robust to small perturbations in the observations, such as could arise from corrupted or missing data.

We consider the following model: There are
$n$ observed binary response variables $(y_1,\dots,y_n)\in \{0,1\}^n$, and we have a design matrix of covariates $\mathbf X = (\mathbf{X}_{i,j})\in \R^{n\times d}$, and write $\mathbf X_{[i]}\in \Rd$ for the $i$th row, $i\in \{1,\dots,n\}$. We consider the response variables and design matrix to be fixed \textit{a priori}.
For a parameter vector $x\in\Rd$ (interpreted as regression coefficients) the likelihood function for the response variables is given as 
\[
\ell(x\mid y_1,\dots,y_n) = 
\prod_{i=1}^n p_i(x)^{y_i} \cdot (1-p_i(x))^{1-y_i},
\]
with
\begin{equation*}
    p_i(x) := \frac{1}{1+\exp (-\mathbf X_{[i]}^\top x)},\quad x\in\Rd.
\end{equation*}
Following \cite{Pollock2020, Kumar2019}, we place an improper uniform prior on the parameters $x$, and obtain the Lebesgue density $\pi$ of the posterior distribution on $\Rd$, which satisfies
\begin{equation}
    \pi(x)\propto \prod_{i=1}^n p_i(x)^{y_i} \cdot (1-p_i(x))^{1-y_i}.
    \label{eq:pi_logistic}
\end{equation}
We assume that the response and design matrix are such that \eqref{eq:pi_logistic} defines a \textit{bona fide} probability density function (that is, the right-hand side of \eqref{eq:pi_logistic} is Lebesgue integrable).

For our stochastic process, we are using a general reversible diffusion of the form \eqref{eq:SDE}, namely
\begin{equation*}
    \dif X_t = \nabla A(X_t)\dif t + \dif W_t,
\end{equation*}
where the drift $A$ is chosen independently of the data. We assume that the relevant technical conditions on $A$ detailed in \cite{Wang2019TheoQSMC} hold, to ensure that $\pi$ is a valid quasi-stationary distribution, for details see also Section~\ref{sec:background}. As in \cite{Wang2019TheoQSMC}, we work on the Hilbert space $\mathcal H = \mathcal L^2(\Rd,\Gamma)$, where $\Gamma(\dif x)=\gamma(x)\dif x$ with $\gamma(x)=\exp(2A(x))$.

Note that the observations $y_1,\dots,y_n$ and covariates $\mathbf X$ only influence the killed diffusion through the killing rate $\kappa$ and so any perturbation of the data results in a perturbation of the killing rate. 

From \eqref{eq:kappa_QSMC} and following the calculations in \cite[Appendix C]{Kumar2019}, we can express the killing rate as
\begin{equation}
\begin{split}
        \kappa(x) = \frac{1}{2}\bigg(  \sum_{j=1}^d\left[ \sum_{i=1}^n \left (y_i-p_i(x)\right )\mathbf X_{i,j} \right]^2 
        - \sum_{j=1}^d \sum_{i=1}^n p_i(x)\left (1-p_i(x)\right )\mathbf X^2_{i,j} 
        \\
        -2\Delta A - 2\left \langle \nabla A, \sum_{i=1}^n \left (y_i - p_i(x)\right ) 
    {\mathbf X_{[i]}}
        \right \rangle \bigg)-\Phi,
    \label{eq:logistic_kappa}
\end{split}
\end{equation}
where the constant $\Phi \in \R$ is chosen to ensure that $\kappa \ge 0$ everywhere; see Condition~(1) of \cite{Pollock2020} for further discussion. 
Let $L^\kappa$ be the generator of the diffusion $X=(X_t)_{t\geq 0}$ killed at rate $\kappa$, which possesses $\varphi = \pi/\gamma$ as its minimal eigenfunction corresponding to the bottom of the spectrum. 
The fact that the eigenspace is one-dimensional follows from the Perron--Frobenius theorem, \cite[Chapter XIII.12]{Reed1978}.
Let us write $\lkap$ for the bottom of the spectrum of $\Lk$. 
Then in fact it follows that $\lkap=-\Phi$; see \cite[Section 3.3]{Wang2019TheoQSMC}; that is, the bottom of the spectrum of the resulting operator is in fact equal to $-\Phi$.

We first consider the common choice of $A\equiv 0$, as in \cite{Pollock2020, Kumar2019}. In this case, the expression \eqref{eq:logistic_kappa} simplifies, and the resulting killing rate $\kappa$ is uniformly bounded from above. 
Moreover, since we are assuming that $\pi$ is integrable, it follows that 
\begin{equation}
\label{eq: implies_spectral_gap}
    \liminf_{\vert x\vert \to \infty} \kappa(x) > \lkap,
\end{equation}
with $\vert \cdot \vert$ denoting the Euclidean norm.

 As already mentioned we imagine having corrupted or missing data, giving us access only to a perturbed killing rate $\widetilde{\kappa} \colon \mathbb{R}^d \to [0,\infty)$.

 \begin{remark}
 For example, it may be the case our covariates are corrupted, so we only have access to $\widetilde{\mathbf X}_{i,j}:=\mathbf X_{i,j}+\epsilon_{i,j}$, where the $(\epsilon_{i,j})\in\mathbb{R}^{n\times d}$ models some corrupting noise (alternatively, the case $\epsilon_{i,j}=-\mathbf X_{i,j}$ could be considered a case of missing data).
 Assume that the corresponding additive constants are equal, $\Phi=\widetilde \Phi$ (which can be done by setting both equal to the maximum value $\Phi \vee \widetilde \Phi$.)
 If we write $\|\mathbf A\|_\infty = \sup_{1\leq i\leq n,\; 1\leq j\leq d}|\mathbf A_{i,j}|$ for a matrix $\mathbf A= (\mathbf{A}_{i,j})\in \mathbb{R}^{n\times d}$, a straightforward calculation shows that for a fixed constant $C>0$ we have \[\|\widetilde\kappa -\kappa\|_\infty \le C n^2 d\|\mathbf X\|_\infty (\|\epsilon\|_\infty^2+\|\epsilon\|_\infty).\] Thus provided the data corruption level $\|\epsilon\|_\infty$ is sufficiently small,  $\|\widetilde\kappa -\kappa\|_\infty$ can be made arbitrarily small.
 \end{remark}
 
 Now we ask whether the difference of the eigenfunctions that determine the quasi-stationary distributions of the killed Markov processes based on $\kappa$ and $\widetilde{\kappa}$ is small when $\Vert \kappa - \widetilde{\kappa} \Vert_\infty$ is sufficiently small. 
 Thus, we consider a perturbation operator $H\colon \mathcal{L}^2(\Rd,\Gamma)\to  \mathcal{L}^2(\Rd,\Gamma)$ given by
    \begin{equation*}
        Hf := (\widetilde \kappa - \kappa)f, \quad f \in \mathcal L^2(\Rd,\Gamma).
    \end{equation*}
 Note that, since $H$ is just a multiplication operator, it is self-adjoint, and the operator norm coincides with $\Vert \kappa - \widetilde{\kappa} \Vert_\infty$. We obtain the following result.
\begin{proposition}  \label{prop:logistic_BM}
    We work in the case $A\equiv 0$ and assume that $\lim_{\|x\|\to\infty}\kappa(x)$ exists. Then there exists a $\delta >0$ and constant $C>0$ such that for any (perturbed) killing rate $\widetilde \kappa$ with $\|\kappa-\tilde \kappa\|_\infty < \delta$, there is a quasi-stationary distribution of the killed Markov process killed at rate $\widetilde\kappa$. In particular, this quasi-stationary distribution has a $\Gamma$-density $\widehat\varphi \in \mathcal L^2(\Rd,\Gamma)$ 
    with $\Vert \widehat{\varphi} \Vert_2 = 1$, which satisfies 
    \begin{equation*}
        \|\varphi-\widehat \varphi \|_2 \le C \|\tilde\kappa-\kappa\|_\infty.
    \end{equation*}
 \end{proposition}
\begin{proof}
We start with arguing that the generator $L^\kappa$ has a spectral gap $\nu>0$. This fact follows readily by \eqref{eq: implies_spectral_gap} and the fact that the base of the essential spectrum must be strictly larger than $\lambda_0^\kappa$ by \cite[Lemma~14.4.1]{Davies2007} (see also \cite[Lemma 3.3(v)]{Kolb2012}, \cite[Remark 1, Section 1.4]{Wang2019TheoQSMC}.
    We choose $\delta$ to be
    \begin{equation*}
        \delta =\frac{1}{2}\min\left \{\nu, \liminf_{\vert x\vert\to \infty}\kappa(x) - \lambda_0^\kappa \right \}.
    \end{equation*}
Let $\|H\|$ be the $\mathcal L^2(\Rd,\Gamma)$-operator norm. Then, we have
\[\|H\| = \Vert \kappa-\widetilde{\kappa} \Vert_\infty
< \frac{1}{2}\left (\liminf_{\vert x\vert \to \infty}\kappa(x) - \lambda_0^\kappa\right).\]
The latter inequality
ensures that the resulting minimal eigenfunction $\widehat \varphi$ is also in $\mathcal L^1(\Rd,\Gamma)$, by \cite[Theorem 4.5]{Champagnat2017}. The fact that the quasi-stationary distribution corresponds to the bottom of the spectrum follows from the argument of \cite{Simon1993}; see also \cite[Section 4.2]{Kolb2012}. The rest follows by applying our abstract result, Theorem~\ref{thm:bdd_case}.
\end{proof}

\begin{remark}
    The bound in Proposition~\ref{prop:logistic_BM} is in terms of the norm $\|\cdot \|_2$ of $\mathcal L^2(\Rd,\Gamma)$ and concerns the functions $\varphi, \widehat{\varphi}$ which are normalized in $\mathcal L^2(\Rd,\Gamma)$, that is, $\|{\varphi}\|_2=\|\widehat{\varphi}\|_2 =1$. This cannot be straightforwardly translated into a bound on the corresponding densities $\pi, \widehat\pi$, which are normalised in $\mathcal L^1(\Rd,\Gamma)$, since in this case $\Gamma(\dif x)=\dif x$ is not a finite measure.
\end{remark}
In light of the above remark, it is of interest to ask whether or not it is possible to construct an alternative diffusion to the case $A\equiv 0$, which still possesses $\pi$ as a valid quasi-stationary distribution, but such that $\Gamma$ is a finite measure. This would enable us, by an application of the Cauchy--Schwarz inequality, to conclude robustness in $\mathcal L^1(\Rd,\Gamma)$ and not only in $\mathcal L^2(\Rd,\Gamma)$. This can be indeed done, by choosing $A$ to satisfy the following properties.
\begin{assumption}
    The function $A:\Rd\to \R$ is a smooth function, and there exists some $M>0, \alpha >0$ such that $A(x)=-\alpha \vert x\vert$ for all $x\in \Rd$ with $\vert x\vert \ge M$, and $\pi/\exp(2A)$ is uniformly bounded from above.
    \label{assm:logistic_A}
\end{assumption}
 
 This is largely a technical assumption for the following proof, although in addition, the resulting diffusion process with drift $\nabla A$ will also be relatively straightforward to simulate in the tails due to its simple form. 
 However, such an $A$ exists since $\pi$ is a log-concave density function.
  With this choice of $A$, we have the following result.
 
 \begin{proposition}
     Assume that the function $A\colon \mathbb{R}^d \to \mathbb{R}$ satisfies Assumption~\ref{assm:logistic_A} and that $\lim_{\|x\|\to\infty}\kappa(x)$ exists. Then, there exists a $\delta >0$ and $C>0$, such that for any perturbed killing rate $\widetilde{\kappa}$ that satisfies $\|\widetilde \kappa - \kappa\|_\infty <\delta$, there is a quasi-stationary distribution $\widehat \pi \in \mathcal L^1(\Rd,\Gamma)$ of the killed Markov process based on $\widetilde{\kappa}$. In particular, we have
     \begin{equation}
         \int_{\mathbb{R}^d} |\pi(x)-\widehat \pi(x)|\dif x \le C \|\tilde \kappa - \kappa\|_\infty. 
         \label{eq:logistic_L1}
     \end{equation}
     \label{prop:logistic_bd}
 \end{proposition}
 \begin{proof}
    Note that by Assumption~\ref{assm:logistic_A} we have that the killing rate $\kappa$ from \eqref{eq:logistic_kappa}, which appears in the unperturbed killed Markov process, is bounded from above. Moreover, \eqref{eq: implies_spectral_gap} 
    is satisfied and therefore there exists similarly by \cite[Lemma~14.4.1]{Davies2007} a spectral gap $\nu>0$ for the generator $L^\kappa$. As in the proof of Proposition~\ref{prop:logistic_BM} it follows that
    \[
        \|\varphi-\widehat \varphi \|_2 \le 
        \widetilde{C} \|\tilde\kappa-\kappa\|_\infty   
    \]
for the corresponding $\mathcal L^2(\Rd,\Gamma)$-normalized eigenfunctions for some $\widetilde{C}>0$. 
    The desired $\mathcal L^1(\Rd,\Gamma)$ bound \eqref{eq:logistic_L1} follows since $\Gamma$ is assumed to be a finite measure. Therefore, we can apply Proposition~\ref{prop:L1_pi_bound} and obtain the assertion.
 \end{proof}

\section{Unbounded perturbation: truncation}
\label{sec:truncation}
In this section we consider a particular unbounded perturbation corresponding to \textit{truncating} the killing rate. This is of practical importance, since for all current implementations of QSMC methods, when using either particles \cite{Pollock2020, DelMoral2003, Burdzy2000} or stochastic approximation \cite{Kumar2019, Mailler2020,Wang2020approx, Aldous1988}, it is necessary to simulate the killing times \eqref{eq:taud}. This simulation is an intricate task when $\kappa$ is unbounded; see \cite{Pollock2020}. On the other hand, when $\kappa$ is uniformly bounded from above, simulation of such killing times is straightforward using Poisson thinning; see \cite{Lewis1979}.

The simulation of \textit{piecewise-deterministic Markov processes} in the context of Monte Carlo \cite{Bierkens2019, Bouchard-Cote2018} also requires the simulation of first event times. In that context, stability investigations with respect to truncation of the event rate have also appeared, in \cite[Section 10]{Durmus2021}.

\subsection{Projection notation}
\label{subsec:projection_notation}

Following \cite{DavisKahan1970}, we make use of projections onto invariant subspaces. Now we describe the appropriate setting and notation, which will be used throughout Section~\ref{sec:truncation}.

For an unbounded self-adjoint operator $(L, \Dom(L))$ on a separable Hilbert space $\H$, we suppose that we have a one-dimensional $L$-invariant subspace, given by its orthogonal projection $P=v_0 \langle \cdot, v_0\rangle$, for some $v_0\in\mathcal{H}$ with $\|v_0\|=1$. In fact $v_0$ must be an eigenvector of $L$, with eigenvalue  $\lambda_0\in\mathbb{R}$, say.
Let 
$P^{\perp}$
denote the orthogonal projection onto the complement subspace $(\lspan\{v_0\})^\perp$.

Following \cite{DavisKahan1970}, we can define isometric mappings $E_0: \mathcal K(E_0) \to \H$ and $E_1: \mathcal K (E_1) \to \H$, with source Hilbert spaces $\mathcal K(E_0), \mathcal K(E_1)$, which satisfy
\begin{equation*}
    E_0 E_0^* = P,\quad E_1 E_1^* = P^{\perp},
\end{equation*}
where $E_0^*: \H \to \mathcal K(E_0)$ and $E_1^*: \H \to \mathcal K(E_1)$ are the corresponding adjoint operators.

Since $P$ projects onto the one-dimensional subspace $\lspan\{v_0\} \subseteq \mathcal{H}$ we can
identify $\mathcal K(E_0) \equiv \R$, equipped with the Euclidean inner product. This allows the explicit characterization 
\begin{equation*}
    \begin{split}
        E_0&: \R \to \mathcal H, \quad \theta \mapsto \theta v_0,\\
        E_0^* &: \mathcal H \to \R,\quad v \mapsto \langle v, v_0\rangle.
    \end{split}
\end{equation*}
By the $L$-invariance of $\lspan\{v_0\}$, we can now decompose $L\colon \mathcal{D}(L) \to \mathcal{H}$ as
\begin{equation*}
    Lv = 
    \begin{pmatrix}
        E_0 & E_1
    \end{pmatrix}
    \begin{pmatrix}
        L_0 & 0 \\
        0 & L_1 
    \end{pmatrix}
    \begin{pmatrix}
        E_0^* v\\
        E_1^*v
    \end{pmatrix},
\end{equation*}
for $v\in \Dom(L)$, where $L_0:\R \to \R$ is the bounded linear operator given by $L_0\theta = \lambda_0 \theta$, and $L_1:\Dom(L_1)\to \mathcal K(E_1)$ is the unbounded operator given by $L_1 = E_1^* L E_1$ with $\Dom(L_1)=E_1^* \Dom(L)$.

Given another self-adjoint operator $(\widehat L, \Dom(\widehat L))$ on the same Hilbert space $\mathcal{H}$, which we will later think of as the \textit{perturbed operator}, we assume that we have a one-dimensional $\widehat L$-invariant subspace given by the orthogonal projection $Q = \widehat v_0 \langle \cdot, \widehat v_0\rangle$, where $\|\widehat v_0\|=1$. So $\widehat v_0$ is an eigenvector of $\widehat L$, with corresponding eigenvalue $\widehat{\lambda}_0\in\mathbb{R}$, say. Let us write $Q^{\perp}$ for the projection onto the orthogonal complement of $\lspan\{\widehat{v}_0\}$.

As above we can similarly define corresponding isometric mappings $F_0:\mathcal \R\to \H$ and $F_1:\mathcal K(F_1)\to \H$ with $F_0 \theta = \theta \widehat{v}_0$ and $F_0^* v = \langle v,\widehat{v}_0 \rangle$, such that $F_0 F_0^* = Q$ and $F_1F_1^*=Q^\perp$. 
Moreover, with these definitions, 
$\widehat{L} \colon \Dom(\widehat{L})\to \mathcal{H} $ can be decomposed for
$v\in \Dom(\widehat L)$ as
\begin{equation}
    \widehat L v = 
    \begin{pmatrix}
        F_0 & F_1
    \end{pmatrix}
    \begin{pmatrix}
        \widehat L_0 & 0 \\
        0 & \widehat L_1 
    \end{pmatrix}
    \begin{pmatrix}
        F_0^*v \\
        F_1^*v
    \end{pmatrix},
    \label{eq:decom_hatA}
\end{equation}
where $\widehat{L}_0\colon \mathbb{R} \to \mathbb{R}$ is given by $\widehat{L}_0\theta = \widehat{\lambda}_0 \theta $ and $\widehat{L}_1 \colon \mathcal{D}(\widehat{L}_1) \to \mathcal{K}(F_1)$ is the unbounded operator given by $\widehat{L}_1 = F_1^* \widehat{L} F_1$ with $\mathcal{D}(\widehat{L}_1) = F_1^* \mathcal{D}(\widehat{L})$.

With these definitions we are able to formulate the remarkable sin~$\theta$ theorem of \cite{DavisKahan1970}. For convenience of the reader, we state it here in our present notation.
\begin{theorem}[The sin~$\theta$ theorem, \cite{DavisKahan1970}]
    Suppose there is some real interval $[a_0, a_1]$ and  $\delta>0$, such that the spectrum of $L_0$ is contained within $[a_0, a_1]$, and the spectrum of $\widehat L_1$ 
    lies
    entirely outside the interval $(a_0-\delta, a_1+\delta)$. Then for $\sin \Theta_0 := \sqrt{1-\langle v_0,\widehat v_0 \rangle^2}$ and operator $R:=\widehat LE_0 - E_0 L_0$, we have $|\sin \Theta_0|\le \|R \|/\delta $.
\end{theorem}
In fact, this is a slight modification of the $\sin \,\theta$ theorem as stated in \cite{DavisKahan1970}, since our operators $L, \widehat L$ are in general unbounded; but see the discussion immediately preceding \cite[Theorem 5.2]{DavisKahan1970}. These subtleties will be addressed in the sequel.

 \subsection{Further notation and results}

We consider the situation described in Section~\ref{sec:background} and set $\H = \mathcal L^2(\Rd, \Gamma)$, which consists of all measurable real-valued functions $f\colon \Rd \to \R $ 
satisfying
\[
    \Vert f \Vert_2:= \left( \int_{\Rd} f(y)^2 \dif\Gamma(y) \right)^{1/2} <\infty,
\]
where the measure $\Gamma$ is given by $\dif \Gamma(y)=\gamma(y)\dif y$ and $\gamma:= \exp(2A)$ with $\dif y$ denoting the Lebesgue measure on $\Rd$. In other words $\H$ is 
the Hilbert space of functions, equipped with the inner product
\begin{equation*}
    \langle f,g\rangle = \int_{\Rd} f(y)g(y) \dif\Gamma(y),
\end{equation*}
where $f,g:\Rd \to \R$. As before by $\mathsf{S}(\H)$ we denote the unit sphere in $\H$ w.r.t. $\Vert\cdot\Vert_2$.

The infinitesimal generator of the diffusion $X$ killed at rate $\kappa$, see \eqref{eq:SDE}, is given by
\begin{equation}
    L^\kappa  =  -\frac{1}{2}\Delta - \nabla A \cdot \nabla + \kappa,
    \label{eq:L^kap}
\end{equation}
where $\Delta$ denotes the Laplacian operator. 
As in \cite[Section 3.3.3]{Wang2020}, the operator $L^\kappa$ can be realised as a densely-defined self-adjoint positive semibounded linear operator on $\H$ with domain
\begin{equation}
    \Dom(L^\kappa) = \left \{f\in \H: -\frac{1}{2}\Delta f - \nabla A \cdot \nabla f + \kappa f \in \H \right\},
    \label{eq:dom_Lkappa}
\end{equation}
where the derivatives are understood in the weak sense. Note that $\Dom(\Lk)$ contains $C_c^\infty(\Rd)$, the set of smooth, compactly supported functions on $\Rd$.
Therefore we have a self-adjoint unperturbed operator on $\H$ given by \eqref{eq:L^kap}
with domain $\Dom(L^\kappa)$. 

For brevity we write $L^0=-\frac{1}{2}\Delta-\nabla A\cdot \nabla$ for the (formal) unkilled differential operator.
We impose the following condition on the killing rate.
\begin{assumption}
    The killing rate is measurable, locally bounded and nonnegative with $\kappa(x)\to \infty$ as $\vert x\vert \to \infty $.
    \label{assm:unbdd_kappa}
\end{assumption}

In the case where $A\equiv0$ and $\kappa$ is given by \eqref{eq:kappa_QSMC}, i.e., in the setting of QSMC with Brownian motion, Assumption~\ref{assm:unbdd_kappa} corresponds to a scenario where $\pi$ has tails decaying at a faster-than-exponential rate (such as Gaussian or lighter).
A standard consequence of the assumption is the following:
\begin{lemma}
Under Assumption~\ref{assm:unbdd_kappa}, there is a strictly positive eigenfunction $\varphi_0 \in \mathsf S(\H)$ of $L^\kappa$ with eigenvalue $\lambda_0^\kappa$ satisfying 
$\lambda_0^\kappa = \inf \sigma(\Lk)$, i.e., it is given by the bottom of the spectrum of $\Lk$. Moreover, $\varphi_0$ corresponds to the unique quasi-stationary distribution of $X$.
\end{lemma}
\begin{proof}
    Existence and uniqueness of the quasi-stationary distribution follows from \cite[Theorem 4.5]{Champagnat2017}. Furthermore, that the function $\varphi_0$ is in $\mathcal L^2$ and can be chosen to be strictly positive follows, for instance, from the fact that under Assumption~\ref{assm:unbdd_kappa}, the operator $L^\kappa$ has a purely discrete spectrum; see \cite[Theorems~XIII.67~and~XIII.44]{Reed1978}.
\end{proof}
\begin{remark}
    Setting $\lambda_1^\kappa := \inf\left \{\sigma(L^\kappa)\backslash \{\lambda^\kappa_0\} \right \}$ we have by the fact that 
$\Lk$ is self-adjoint and bounded from below, see \cite[Theorem XIII.1]{Reed1978}, the variational representations as in \eqref{eq:var_eigenv}, that is,
    \begin{equation*}
        \begin{split}
            \lkap &= \inf_{f\in \mathsf S(\H)\cap \Dom(\Lk)} \langle f,\Lk f\rangle,\\
            \lambda_1^\kappa &= \sup_{\mathbb V \in \mathcal V_1} \inf_{f\in \mathbb V^\perp \cap \mathsf S(\H)\cap \Dom(\Lk)} \langle f, \Lk f\rangle.
        \end{split}
    \end{equation*} 
\end{remark}

We are interested in truncating the unbounded killing rate $\kappa$ and understanding the resulting impact on the quasi-stationary distribution. So, we define, for arbitrary $M\ge 0$ the \textit{truncated killing rate} $\kappa_M:\Rd\to [0,\infty)$; writing $a\wedge b := \min\{a,b\}$, it is given pointwise by
\begin{equation*}
    \kappa_M(x) = \kappa(x) \wedge M.
\end{equation*}
With this we can define our perturbed operator, $\widehat{L}^M :=L^{\kappa_M}$ on $\H$, i.e., 
\begin{equation}
    \widehat{L}^M = -\frac 1 2 \Delta - \nabla A \cdot \nabla + \kappa_M.
    \label{eq: ger_pert}
\end{equation}
Similarly to $\Lk$, $\widehat{L}^M$ can also be realised as a self-adjoint operator, with corresponding domain $\Dom (\widehat{L}^M)$ given analogously to \eqref{eq:dom_Lkappa}. This corresponds to our diffusion $X$ killed at state-dependent rate $\kappa_M$.

We require some further notation for stating the resulting perturbation. For any function $f:\Rd \to \R$, we write $f_+$ for the positive part; $f_+(x)=\max\{f(x),0\}$ for each $x\in \Rd$. Define now our (unbounded) perturbation $H_M$ on $\H$ as the multiplication operator given by
\begin{equation*}
    H_M f = -(\kappa -M)_+f,
\end{equation*}
for $f \in \Dom(H_M)$ with $\Dom(H_M):=\{f\in \H: \int_{\Rd} |(\kappa - M)_+ f|^2 \dif \Gamma <\infty\} $. 
We can characterize $\Dom(H_M)$.
\begin{lemma}
    For any $M\geq 0$ we have
    \begin{equation}  \label{eq: char_dom_H_M}
        \Dom(H_M) = \{f\in \H: \int_{\Rd}|\kappa f|^2\dif \Gamma <\infty\}. 
    \end{equation}
    Consequently, $\Dom(H_M)$ is independent of $M$.
\end{lemma}
\begin{proof}
    The statement of \eqref{eq: char_dom_H_M} follows immediately by observing that $|(\kappa-M)_+f - \kappa f|\le M|f|$ holds pointwise.
\end{proof}
To emphasize that the domain of $H_M$ does not depend on $M$ we define $\Dom(\kappa):= \Dom(H_M)$.
\begin{lemma}\label{lemma:dom_LM}
    Suppose $g\in \Dom(\Lk)\cap\Dom(\kappa)$. 
    Then $g\in \Dom(\widehat{L}^M)$ and
    \begin{equation*}
        \widehat{L}^M g = L^\kappa g + H_M g.
    \end{equation*}
\end{lemma}
\begin{proof}
    We take similar steps to the proof of \cite[Lemma 15]{Wang2020+}. For a sequence of mollifiers $(f_n)_{n\in\mathbb{N}}$ as in the proof, similarly consider $g_n := f_n g$. Since $g \in \Dom(L^\kappa)$, it is twice weakly differentiable. Then each $g_n$ is a compactly supported, twice weakly differentiable function. Furthermore, $g_n \to g$ in the Sobolev space $W^{2,1}(\Rd)$, and we have that $|g_n| \le C|g|$ for some constant $C$, pointwise, and uniformly over $n$.
    
    Hence, using the fact that $L^\kappa$ is closed and $\kappa g \in \mathcal H$, and the dominated convergence theorem, we have the following convergence in $\mathcal H$,
    \begin{equation*}
        \widehat{L}^M g_n = \widehat{L}^0 g_n + \kappa g_n - (\kappa-M)_+ g_n  \longrightarrow L^\kappa g - (\kappa-M)_+ g \in \mathcal H.
    \end{equation*}
\end{proof}
From now on, we impose the following standard assumption.
\begin{assumption}
    The eigenfunction $\varphi_0$ of $L^\kappa$ with eigenvalue $\lambda_0^\kappa$ satisfies $\varphi_0 \in \Dom(\kappa)$, that is, $\kappa \varphi_0 \in \H$.
    \label{assm:varphi_kappa_integrable}
\end{assumption}
\begin{remark}
    In the QSMC context with Brownian motion, i.e., $A\equiv 0$, Assumption~\ref{assm:varphi_kappa_integrable} corresponds to requiring that $\Delta \pi$ is square-integrable, which is certainly true for densities of the form $\exp(-\vert x\vert^\alpha)$ with $\alpha \ge 2$.
\end{remark}
Together Lemma~\ref{lemma:dom_LM} and Assumption~\ref{assm:varphi_kappa_integrable} immediately imply the following.
\begin{lemma}
    Under Assumption~\ref{assm:varphi_kappa_integrable} we have $\varphi_0\in \Dom(\widehat{L}^M)$.
    \label{lemma:varphi0_domLM}
\end{lemma}
We now turn to the question of existence of a perturbed quasi-stationary distribution, corresponding to a minimal eigenfunction of $\widehat{L}^M$. Clearly, if we set $M=0$, there is no killing, and so there is no quasi-stationary distribution to speak of. Thus it is necessary to enforce a minimum truncation level.

\begin{assumption}
    The truncation level satisfies $M \ge  \lkap$.
    \label{assm:M>lkap}
\end{assumption}
\begin{theorem}
    Under Assumptions~\ref{assm:unbdd_kappa}, \ref{assm:varphi_kappa_integrable} and \ref{assm:M>lkap}, the 
    diffusion $X$ killed at rate $\kappa_M$ 
    with generator $\widehat{L}^M$ possesses a unique quasi-stationary distribution, and there is a corresponding strictly positive eigenfunction $\widehat \varphi_0 \in \H$ of $\widehat{L}^M$ with (minimal) eigenvalue $\widehat{\lambda}_0^M := \inf \sigma(\widehat{L}^M)$.
    \label{thm:tildvarphi0}
\end{theorem}
\begin{proof}
    Since $\widehat{L}^M$ is a self-adjoint semibounded operator, we have the variational representation \cite[Theorem XIII.1]{Reed1978},
    \begin{equation*}
        \widehat{\lambda}_0^M = \inf_{f \in \Dom(\widehat{L}^M)\cap \mathsf S(\H)} \langle f, \widehat{L}^M f\rangle.
    \end{equation*}
    From Lemma~\ref{lemma:varphi0_domLM}, we have $\varphi_0 \in \Dom(\widehat{L}^M)\cap \mathsf S(\H)$. Hence, making use of Lemma~\ref{lemma:dom_LM}, we have
    \begin{equation*}
        \begin{split}
            \widehat{\lambda}_0^M &\le \langle \varphi_0, \widehat{L}^M \varphi_0\rangle \\
            &= \langle \varphi_0, \Lk \varphi_0\rangle+\langle \varphi_0, H_M \varphi_0\rangle\\
            &= \lkap -\int_{\Rd} (\kappa-M)_+ \varphi_0^2\dif \Gamma < M,
        \end{split}
    \end{equation*}
    since we are assuming that $\lkap \le M$, and the final term is strictly negative, since $\varphi_0$ is strictly positive.
    
    Thus the bottom of the spectrum $\widehat{\lambda}_0^M$ of $\widehat{L}^M$ is strictly smaller than $M$, which is the limit of the killing at infinity, hence by \cite[Theorem 4.5]{Champagnat2017} (and \textit{c.f.} \cite[Theorem 4.3]{Kolb2012}), there is a unique quasi-stationary distribution, and there is a corresponding eigenfunction $\widehat \varphi_0$ with eigenvalue $\widehat{\lambda}_0^M$.
\end{proof}
\begin{remark}
    The constant $\lkap$ corresponds to the constant $-\Phi$ in Condition~(1) of \cite{Pollock2020}. It can also be expressed as the expected value of $\kappa$ w.r.t. distribution $\pi$.
\end{remark}
\begin{remark}
    It follows, by \cite[Lemma~14.4.1]{Davies2007} that for the bottom of the essential spectrum, we have $\inf \sigma_\mathrm{ess}(\widehat{L}^M)\ge M$, and hence $\widehat{L}^M$ also possesses a spectral gap. This can also be read off from the proof of Theorem~\ref{thm:tildvarphi0}. The calculation there shows that
    \begin{equation*}
        \widehat{\lambda}_0^M < \lambda_0^\kappa.
    \end{equation*}
\end{remark}

    Now we define for arbitrary $M\geq 0$,
\begin{equation*}
    \widehat\lambda_1^M := \sup_{\mathbb V \in \mathcal V_1} \inf_{f\in \mathbb V^\perp \cap \mathsf S(\H)\cap \Dom(\widehat L^M)} \langle f, \widehat  L^M f\rangle.
\end{equation*}
Then by \cite[Theorem XIII.1]{Reed1978}, $\widehat  \lambda_1^M$ is either an eigenvalue of $\widehat  L^M$ or coincides with the bottom of the essential spectrum $\sigma_\mathrm{ess}(\widehat  L^M)$.

\begin{proposition}
    Fix $M_2>M_1>0$. Then $\lambda_1^\kappa\ge\widehat \lambda_1^{M_2} \ge \widehat \lambda_1^{M_1}$.
    \label{prop:monoton_eval}
\end{proposition}
\begin{proof}
    This follows from a routine calculation, which we include for completeness. Since the set of smooth, compactly supported functions $C_\mathrm c^\infty(\Rd)\subset \H$ is a core for $\Lk$ and for $\widehat L^M$ for any $M>0$, in the variational representation of eigenvalues we can restrict to functions $f\in C_\mathrm c^\infty(\Rd)$.
    Then,
    \begin{equation*}
        \begin{split}
            \widehat \lambda_1^{M_2} &= \sup_{\mathbb V \in \mathcal V_1} \inf_{f\in \mathbb V^\perp \cap \mathsf S(\H)\cap C_\mathrm c^\infty}
            \left (\langle f,\widehat L^0 f\rangle + \langle f,(\kappa \wedge M_2)f\rangle \right )\\
            &= \sup_{\mathbb V \in \mathcal V_1} \inf_{f\in \mathbb V^\perp \cap \mathsf S(\H)\cap C_\mathrm c^\infty}
            \left ( \langle f,\widehat L^0 f\rangle + \langle f,(\kappa \wedge M_1)f\rangle + \langle f, (\kappa\wedge M_2 - \kappa \wedge M_1)f\rangle \right)\\
            &\ge \sup_{\mathbb V \in \mathcal V_1} \inf_{f\in \mathbb V^\perp \cap \mathsf S(\H)\cap C_\mathrm c^\infty}
            \left ( \langle f,\widehat L^0 f\rangle + \langle f,(\kappa \wedge M_1)f\rangle  \right)\\
            &= \widehat \lambda_1^{M_1}.
        \end{split}
    \end{equation*}
    An analogous argument also 
    gives
    $\lambda_1^\kappa\ge\widehat \lambda_1^{M_2}$.
\end{proof}

We have established in Theorem~\ref{thm:tildvarphi0} the existence of a minimal eigenfunction $\widehat \varphi_0$ of $\widehat L^M$, under Assumptions~\ref{assm:unbdd_kappa}, \ref{assm:varphi_kappa_integrable} and \ref{assm:M>lkap} which we assume to hold in the following. We can thus define one-dimensional invariant subspaces of $\Lk$ and $\widehat L^M$ respectively, $\lspan\{\varphi_0\}$ and $\lspan\{\widehat  \varphi_0\}$, and make use of the projection operators 
as previously defined in Section~\ref{subsec:projection_notation}. 
Thus, with $E_0: \R \to \H$ and $E_1: \mathcal K(E_1)\to \H$ we have for $f\in \Dom(\Lk)$ that
\begin{equation*}
    \Lk f=    
    \begin{pmatrix}
        E_0 & E_1
    \end{pmatrix}
    \begin{pmatrix}
        \Lambda_0 & 0\\
        0 & \Lambda_1
    \end{pmatrix}
    \begin{pmatrix}
        E_0^*f\\
        E_1^*f
    \end{pmatrix}.
\end{equation*}
In particular, $\Lambda_0:\R\to\R$ is the bounded linear operator given by
\begin{equation*}
    \Lambda_0: a \mapsto \lkap a, \quad a\in\R,
\end{equation*}
and $\Lambda_1$ is the self-adjoint operator given by $\Lambda_1=E_1^* \Lk E_1$ with domain $\Dom(\Lambda_1) = E_1^* \Dom(\Lk)$.
For our perturbed operator $\widehat L^M$, we have with $F_0: \R\to \H$ and $F_1: \mathcal K(F_1)\to \H$ the decomposition \begin{equation*}
    \widehat L^M f=
    \begin{pmatrix}
        F_0 & F_1
    \end{pmatrix}
    \begin{pmatrix}
        \widehat \Lambda_0 & 0 \\
        0 & \widehat \Lambda_1 
    \end{pmatrix}
    \begin{pmatrix}
        F_0^*f \\
        F_1^*f
    \end{pmatrix},
\end{equation*}
for $f \in \Dom(\widehat L^M)$, where $\widehat{\Lambda}_0 \colon \R \to 
\R$ with $\widehat{\Lambda}_0 a = \widehat{\lambda}_0^M a$ and
$\widehat \Lambda_1=F_1^* \widehat L^M F_1$ are self-adjoint operators. Note that throughout, the norm of an operator always refers to the corresponding operator norm.

Our goal is to bound $\|\varphi_0-\widehat\varphi_0\|_2$.
For this we aim to apply the results of \cite{DavisKahan1970}.
However, we cannot immediately use the sin~$\theta$ theorem as described in \cite{DavisKahan1970}, because of the unbounded nature of our operators.
In order to overcome this issue, we begin by defining an operator $R: \R \to \H$, given by 
\begin{equation}
\label{eq: R}
    R = \widehat L^M E_0 - E_0 \Lambda_0.
\end{equation}
\begin{lemma}
The operator $R \colon \R\to\H$ as defined in \eqref{eq: R} is a bounded linear operator, which can be represented as
    \begin{equation*}
        Ra = H_M E_0 a = a (H_M\varphi_0), \quad a\in\R.
    \end{equation*}
    Moreover, the adjoint operator $R^*: \H \to \R$ is given by
    \begin{equation}
        R^* f = \langle f, H_M\varphi_0\rangle, \quad f\in \H.
        \label{eq:rstar}
    \end{equation}
\end{lemma}
\begin{proof}
    Since $\mathrm{Ran}(E_0):= \{E_0 a:a\in \R\} =\lspan\{\varphi_0\}$, by using Lemma~\ref{lemma:varphi0_domLM} we obtain that $R$ is well-defined. Thus $\widehat L^M E_0$ is a bounded operator, and since $E_0 \Lambda_0$ is also a bounded operator, $R$ is a bounded operator.
    
    Furthermore, we can directly compute that
    \begin{equation*}
        \begin{split}
            Ra &= \widehat L^M a \varphi_0 - \lambda_0^\kappa a \varphi_0\\
            &= a [L^\kappa \varphi_0 + H_M\varphi_0- \lambda_0^\kappa  \varphi_0]\\
            &= a[ \lambda_0^\kappa  \varphi_0+H_M\varphi_0- \lambda_0^\kappa \varphi_0]
            = a H_M \varphi_0,
        \end{split} 
    \end{equation*}
    where we have made use of Lemma~\ref{lemma:dom_LM}.
    
    For the adjoint, recall that $R^*: \H \to \R$ is the unique operator with the property that for arbitrary $f\in\mathcal H$ and $a \in \R$, we have that
    \begin{equation*}
        \langle f, R a \rangle =  a (R^* f) .
    \end{equation*}
    A simple calculation then gives \eqref{eq:rstar}.
\end{proof}

Consider now the bounded operator $R^* F_1: \mathcal K(F_1) \to \R$. By \eqref{eq:rstar}, we have for $h \in \mathcal K(F_1)$,
\begin{equation*}
    R^* F_1 h = \langle F_1 h, H_M \varphi_0 \rangle.
\end{equation*}

\begin{lemma}
 The bounded operator $R^* F_1$ is the closure of the densely-defined operator 
    \begin{equation*}
        E_0^* F_1 \widehat\Lambda_1 - \Lambda_0 E_0^* F_1,
    \end{equation*}
    on smooth compactly-supported functions.
    \label{lemma:R=E0F}
\end{lemma}
\begin{proof}
    Let $f\in \H$ be a smooth compactly-supported function and consider $h := F_1^* f\in \mathcal K(F_1)$. Let $Q^\perp=F_1^*F_1$ be the projection on $\lspan\{\widehat\varphi_0\}^\perp$ as well as $P$ be the projection onto $\lspan\{\varphi_0\}$ and $P^{\perp}$ the projection onto $\lspan\{\varphi_0\}^\perp$. We can straightforwardly compute, using the fact that $\widehat L^M$ is invariant on $\mathrm{Ran}(F_1)$, and the fact that $\Lk$ is invariant on $\lspan\{\varphi_0\}$ and its orthogonal complement:
    \begin{equation*}
        \begin{split}
            E^*_0 F_1 \widehat\Lambda_1 h - \Lambda_0 E_0^* F_1 h &= E_0^* (Q^\perp\widehat L^M F_1 h) - \Lambda_0 E_0^* F_1 h\\
            &= E_0^* (\widehat L^M F_1 h) - \Lambda_0 E_0^* F_1 h\\
            &= \langle \widehat L^M F_1 h - \lambda_0^\kappa F_1 h , \varphi_0\rangle \\
            &= \langle L^\kappa F_1 h + H_M F_1 h -  \lambda_0^\kappa F_1 h , \varphi_0\rangle\\
            &= \langle L^\kappa (P F_1 h +  P^\perp F_1 h) + H_M F_1 h-  \lambda_0^\kappa F_1 h , \varphi_0\rangle\\
            &= \langle \lambda_0^\kappa PF_1 h +0 + H_M F_1 h -  \lambda_0^\kappa F_1 h , \varphi_0\rangle\\
            &= \lambda_0^\kappa\langle F_1 h,\varphi_0\rangle \langle \varphi_0,\varphi_0\rangle  +\langle H_M F_1 h, \varphi_0\rangle - \langle\lambda_0^\kappa F_1 h , \varphi_0\rangle\\
            &= \langle F_1 h, H_M\varphi_0\rangle
            = R^* F_1 h.
        \end{split}
    \end{equation*}
    At the end we also used self-adjointness of $H_M$ and \eqref{eq:rstar}.
\end{proof}
We require an additional technical condition.
\begin{assumption}
    There exists some $\delta >0$ such that $\widehat \lambda_1^M >\lkap+\delta$.
    \label{assm:spec_sec_Lambda1}
\end{assumption}
The following auxiliary estimate is a consequence of the latter assumption, Lemma~\ref{lemma:R=E0F}, i.e., $R^* F_1 = E_0^* F_1 \widehat \Lambda_1 - \Lambda_0 E_0^* F_1$ on a dense subset of $\H$, and \cite[Theorem 5.1]{DavisKahan1970}; see their comments preceding \cite[Theorem 5.2]{DavisKahan1970} concerning the unbounded case.
\begin{lemma}
    Suppose that Assumption~\ref{assm:spec_sec_Lambda1} holds. Then 
    \begin{equation}
        \delta \|E_0^* F_1\|\le \|R^* F_1\|.
        \label{eq:bdd_Rdelta}
    \end{equation}
\end{lemma}

Now we state our perturbation result, whose proof follows that of the sin~$\theta$ theorem of \cite{DavisKahan1970}.

\begin{theorem}
    Suppose that Assumption~\ref{assm:spec_sec_Lambda1} holds and without loss of generality let $\widehat \varphi_0$ be chosen in such a way that $\langle \varphi_0, \widehat \varphi_0\rangle \geq 0$. Then
    \begin{equation*}
        \|R\|\ge \delta\, \sqrt{ 1- \langle \varphi_0, \widehat \varphi_0\rangle^2 }.
    \end{equation*}
\end{theorem}
\begin{proof}
     We set $\sin \Theta_0 := \sqrt{1-\langle \varphi_0, \widehat \varphi_0\rangle^2}$. As in the proof of the sin~$\theta$ theorem of \cite{DavisKahan1970}, our operator $E^*_0 F_1$ corresponds to a rotation from $\widehat \varphi$ to $\varphi$, and thus its angle corresponds to $\sin \Theta_0$. Hence, making use of \eqref{eq:bdd_Rdelta}, we see that
     \begin{equation*}
        \delta |\sin\Theta_0| = \delta \|E^*_0 F_1\| \le\|R^* F_1\| \le \|R^*\| = \|R\|.
     \end{equation*}
\end{proof}
We can then complete the argument as in the proof of Theorem~\ref{thm:bdd_case} to get the following bound.

\begin{theorem}
Again with $\widehat \varphi_0$ be chosen in such a way that $\langle \varphi_0, \widehat \varphi_0\rangle \geq 0$, then under Assumptions~\ref{assm:unbdd_kappa}, \ref{assm:varphi_kappa_integrable}, \ref{assm:M>lkap} and~\ref{assm:spec_sec_Lambda1} we have
    \begin{equation*}
    \begin{split}
        \Vert \widehat \varphi_0 - \varphi_0\Vert_2 \le \frac{\sqrt 2 }{\sqrt{1+\langle \varphi_0, \widehat \varphi_0\rangle}} \cdot \frac{\|H_M E_0\|}{\delta} \le \frac{\sqrt 2 \,\|H_M E_0\|}{\delta}.
    \end{split}
    \end{equation*}
    \label{thm:unbdd_case}
\end{theorem}
Note that
\begin{equation*}
    \|H_M E_0\| = \|H_M \varphi_0\|_2,
\end{equation*}
where on the left-hand side one has the corresponding operator norm and on the right-hand side one has the Hilbert space norm.
Moreover, observe that since we have Assumption~\ref{assm:varphi_kappa_integrable}, as the truncation $M\to \infty$, we obtain $\|H_M \varphi_0\|_2 \to 0$ and at the same time by Proposition~\ref{prop:monoton_eval}, we have that $\delta$ remains constant with $M$.

\begin{remark}
    When $\Gamma$ is a finite measure, then the $\mathcal{L}^2(\mathbb{R}^d,\Gamma)$--norm based perturbation bound from Theorem~\ref{thm:unbdd_case} can also be translated into an $\mathcal{L}^1(\mathbb{R}^d,\Gamma)$--norm based perturbation bound, analogously to our previous Proposition~\ref{prop:logistic_bd}. We refer to our subsequent Proposition~\ref{prop:unbdd_L1} for an application of this procedure.
\end{remark}

\subsection{Example: unbounded case}
\label{sec:example_unbdd}

We now illustrate our former truncation result on a Gaussian example. In this setting, many relevant quantities are analytically available, so we can carefully isolate the effect of the truncation. We consider a similar example to \cite{Wang2019TheoQSMC}, namely an Ornstein--Uhlenbeck (OU) diffusion targeting a Gaussian quasi-stationary distribution. We will see that in any dimension $d\ge 1$, for a truncation level $M$, the error decays like $\exp(-M)$ up to some (dimension-dependent) multiplicative constant. In the univariate setting $d=1$, we will be able to explicitly compute this constant. Thus truncating the killing rate in this example appears to be robust with respect to dimension, in that the error decays exponentially for any dimension.

Throughout, we write $\mathcal N(\mu, \sigma^2)$ for the univariate Gaussian distribution with mean $\mu\in\R$ and variance $\sigma^2>0$ and $\mathcal N(\mu, \Sigma)$ for the corresponding multivariate Gaussian distribution with mean vector $\mu \in \Rd$ and covariance matrix $\Sigma \in \R^{d\times d}$. Moreover, $I_d$ denotes the identity matrix on $\Rd$.
We take the diffusion to be a multivariate OU process on $\Rd$ determined by
\begin{equation}
    \dif X_t = \frac{1}{2}(-X_t) \dif t + \dif W_t,
    \label{eq:OU}
\end{equation}
which possesses an invariant measure $\Gamma(\dif x) = \exp{(-\vert x\vert ^2/2)}\dif x$ proportional to the $\mathcal N(0,I_d)$-distribution.

We take the target to be $\pi=\mathcal N(0,I_d/2)$, in which case the killing rate has a particularly simple form, 
\begin{equation}
    \kappa(y) = \vert y\vert ^2, \quad y \in \Rd.
    \label{eq:killing_rate_OU}
\end{equation}
It follows by \cite{Wang2019TheoQSMC} that $\pi$ is the quasi-stationary distribution of the killed Markov
process based on $X=(X_t)_{t\geq0}$, see also Section~\ref{sec:background}. 
We are now interested to derive quantitative bounds on the effect of truncating this divergent killing rate at a level $M$. In order to obtain explicit constants, we now focus on the univariate case $d=1$, where all necessary calculations are analytically tractable. We will return to the general multivariate case in Proposition~\ref{prop:multivar_trunc}.

Having this, the normalized (in $\mathcal L^2(\R,\Gamma)$) eigenfunction $\varphi_0\colon\R\to\R$ of the corresponding unperturbed generator $L^\kappa$ is given by
\begin{equation*}
    \varphi_0(x) = C \mathrm e^{-{x^2}/{2}},
\end{equation*}
with $C=(\frac{3}{2\pi})^{1/4}$.

Moreover, the bottom of the spectrum of the unperturbed operator is given by
\begin{equation}
    \lambda_0^\kappa = \frac{1}{2},
\end{equation}
as in \cite[(2.4)]{Wang2019TheoQSMC}, and the entire spectrum can be represented as
\begin{equation}
    \sigma(L^\kappa) = \{\lambda_0^\kappa+ n \sigma_\mathrm{gap}:n\in \mathbb N_0\},
    \label{eq:unperturbed_spec}
\end{equation}
with spectral gap 
\begin{equation*}
    \sigma_\mathrm{gap} = \frac{3}{2},
\end{equation*}
see \cite{Metafune2002}.

Clearly our killing rate $\kappa$ as defined in \eqref{eq:killing_rate_OU} is unbounded, and we are interested to study the effect of truncating this killing rate upon the quasi-stationary distribution. 
For $M\geq 0$ we consider a truncation at level $M$ and set
\begin{equation*}
    \kappa_M(x) := \kappa(x) \wedge M, \quad x\in\R.
\end{equation*}
This gives rise to our perturbed Markov process with killing rate $\kappa_M$ with generator $\widehat{L}^M$. Note that we use the notation of the previous section.

In order to utilise 
the estimate of Theorem~\ref{thm:unbdd_case}
to obtain a quantitative result, we need to calculate the $\delta$-separation between $\lkap$ and $\widehat \lambda_1^M$ as in Assumption~\ref{assm:spec_sec_Lambda1}, since $\delta$ appears in the bounds. 
Calculating the spectrum in general is a difficult problem, so we resort to a scheme which was utilised in \cite{Mielnik1996} to numerically calculate the energy levels of the quantum truncated harmonic oscillator (which correspond to a killed Brownian motion). Since we are interested in a killed Ornstein--Uhlenbeck process \eqref{eq:OU} rather than a killed Brownian motion, we extend their technique into our present setting.

Considering \eqref{eq: ger_pert} we obtain the eigenvalue problem
\begin{equation}
    -\frac{1}{2}\varphi''(x) + \frac{x}{2} \varphi'(x) + \kappa_M(x) \varphi(x) = \lambda \varphi(x),
    \label{eq:eval_prob}
\end{equation}
that is, we seek $\lambda\in\R$ and $\varphi\in\mathcal{L}^2(\R,\Gamma)$ such that \eqref{eq:eval_prob} is satisfied. 
By the fact that the killing rate $\kappa_M$ is constant for $|x|>\sqrt M$, using \cite{Taylor1989}, we can explicitly write down the form of the eigenfunctions in the tails. Define
\begin{equation*}
     K_\beta^\pm(x) := \int_0^\infty r^{\beta -1} \exp(-r^2 \pm \sqrt 2 rx) \dif r, \quad x \in \R.
\end{equation*}
Then given $\lambda\in\R$, the solution of \eqref{eq:eval_prob} when $|x|>\sqrt M$ is
\begin{equation}
    \varphi(x) = \varphi^{(\lambda)} (x) = 
    \begin{cases}
        c_1 K_{2(M-\lambda)}^+(x), \quad x< -\sqrt M, \\
        c_1 K_{2(M-\lambda)}^- (x),\quad x> \sqrt M,
    \end{cases}
    \label{eq:varphi_tails}
\end{equation}
since the eigenfunction must be square integrable. Note that here we have a common constant $c_1>0$ by symmetry. 
Now, following \cite{Mielnik1996}, we  substitute
\begin{equation*}
    q:= \varphi ,\quad p:= \varphi',
\end{equation*}
and change to polar coordinates: for suitable functions $\rho\colon\R\to(0,\infty)$ and $\alpha\colon\R \to \R$ we have
\begin{equation*}
    q=\rho \cos \alpha ,\quad p = \rho \sin \alpha.
\end{equation*}
With this substitution, we can modify \eqref{eq:eval_prob} to the following equations,
\begin{align}
        \rho' \cos \alpha -  \alpha' \rho \sin \alpha &= \rho \sin \alpha ,\label{eq:rho_1}\\ 
        \rho' \sin \alpha + \alpha' \rho \cos \alpha &=2(\kappa_M-\lambda) \rho \cos \alpha + x \rho \sin\alpha \label{eq:rho_2}.
\end{align}
Note that in contrast to \cite[(17)]{Mielnik1996} we require the additional final term in \eqref{eq:rho_2}.

Then by multiplying \eqref{eq:rho_1} by $-\sin \alpha$ and adding it to $\cos \alpha$ times \eqref{eq:rho_2}, we obtain a first-order differential equation for $\alpha$, given by
\begin{equation}
    \alpha' = 2(\kappa_M-\lambda) \cos^2 \alpha -\sin^2 \alpha +x \sin \alpha \cos \alpha.
    \label{eq:alpha_ODE}
\end{equation}
Noting that
\begin{equation*}
    \frac\dif {\dif x}K_\beta^\pm(x) = \pm \sqrt 2\int_0^\infty r^\beta \exp(-r^2 \pm \sqrt 2 rx)\dif r,
\end{equation*}
we can compute that in the tails,
\begin{equation}
    \tan \alpha(x) = \frac{p(x)}{q(x)}=
    \begin{cases}
        \frac{\sqrt 2 \int_0^\infty r^{\beta^*} \exp(-r^2 + \sqrt 2 rx)\dif r }{\int_0^\infty r^{{\beta^*}-1} \exp(-r^2 + \sqrt 2 rx)\dif r} , \quad x\le -\sqrt M, \\
        \frac{-\sqrt 2 \int_0^\infty r^{\beta^*} \exp(-r^2 - \sqrt 2 rx)\dif r }{\int_0^\infty r^{{\beta^*}-1} \exp(-r^2 - \sqrt 2 rx)\dif r},\quad x\ge  \sqrt M,
    \end{cases}
    \label{eq:tan_alpha}
\end{equation}
with ${\beta^*} = 2(M-\lambda)$. 

Since we know the evolution of $\alpha(x)$ is given by \eqref{eq:alpha_ODE}, and we have boundary values given by \eqref{eq:tan_alpha}, in order for $\lambda$ to be a valid eigenvalue, we require the following compatibility condition, analogously to \cite{Mielnik1996}: For some integer $n$,
\begin{equation*}
    \alpha(\sqrt M) = \alpha(-\sqrt M)+ \int_{-\sqrt{M}}^{\sqrt{M}}   \alpha'(x) \dif x +n \pi,
\end{equation*}
where $\alpha'$ is defined by \eqref{eq:alpha_ODE}, and $\alpha (\pm \sqrt M)$ are defined by \eqref{eq:tan_alpha}.

Thus, again following the approach of \cite{Mielnik1996}, we define
\begin{equation}
    s_M(\lambda):= \alpha(\sqrt M) - \alpha(-\sqrt M)- \int_{-\sqrt{M}}^{\sqrt{M}}  \alpha'(x) \dif x,
    \label{eq:s_M_def}
\end{equation}
where the right-hand side can be computed numerically to arbitrary accuracy using a numerical ODE solver. Thus we need to find values of $\lambda$ such that
\begin{equation}
    s_M(\lambda) = n\pi,
    \label{eq:gamma_ME}
\end{equation}
for integers $n$. The corresponding values of $\lambda$ for which \eqref{eq:gamma_ME} hold are precisely the eigenvalues for the truncated operator with truncation level $M$.
We have plotted the function $\lambda\mapsto s_M(\lambda)$ in the cases $M=2,10,40$ in Figure~\ref{fig:gamma}.

\begin{figure}
    \centering
\begin{subfigure}{6cm}
   \includegraphics[width=6cm]{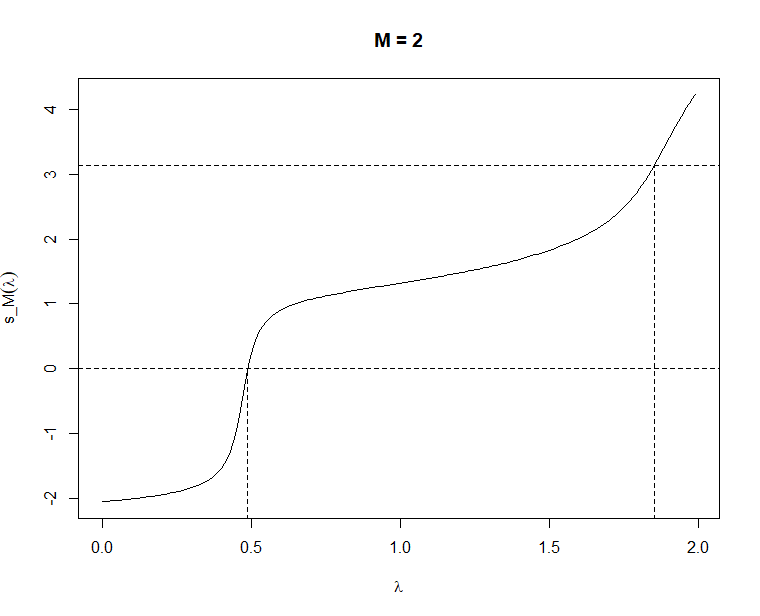} 
    \caption{Here $M=2$ with numerically calculated $\widehat\lambda_0^M \approx 0.4879$ and $\widehat\lambda_1^M \approx 1.8501$. This indicates that we may apply Theorem~\ref{thm:unbdd_case} with $\delta = 1.3$. }
    \label{fig:gamma_E_M2}
\end{subfigure}
\hfill
\begin{subfigure}{6cm}
\includegraphics[width=6cm]{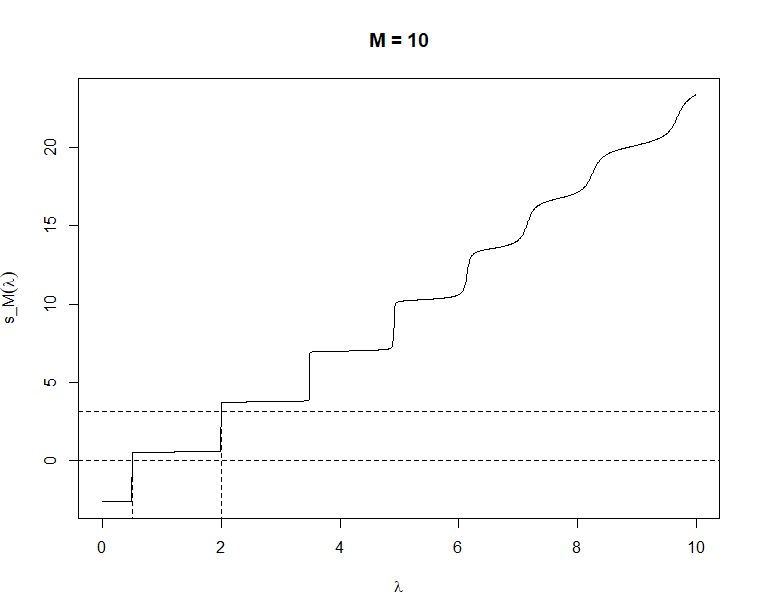}
\caption{Here $M=10$ with numerically calculated $\widehat\lambda_0^M \approx 0.4999$ and $\widehat\lambda_1^M \approx 1.9984$. This indicates that we may apply Theorem~\ref{thm:unbdd_case} with $\delta = 1.4$.}
    \label{fig:gamma_E_M10}
\end{subfigure}
\hfill
\begin{subfigure}{8cm}
    \includegraphics[width=8cm]{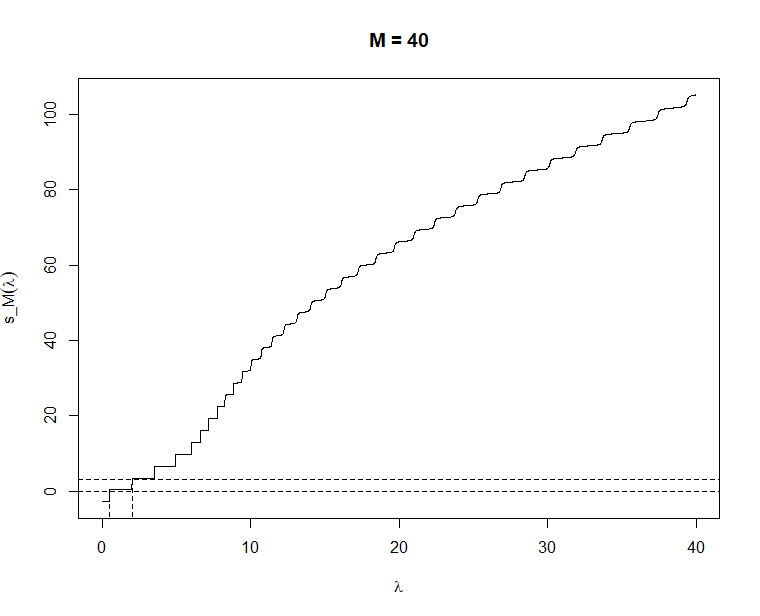}
  \caption{Here $M=40$ numerically calculated $\widehat\lambda_0^M \approx 0.49999$ and $\widehat\lambda_1^M \approx 1.99842$. This indicates that we may apply Theorem~\ref{thm:unbdd_case} with $\delta = 1.4$.}
    \label{fig:gamma_E_M40}
\end{subfigure}
\caption{Plots of the function $\lambda\mapsto s_M(\lambda)$, as defined in \eqref{eq:s_M_def}, for $M=2,10,40$.}
\label{fig:gamma}
\end{figure}

For analogous plots in the case of Brownian motion (rather than our Ornstein--Uhlenbeck process \eqref{eq:OU}), see Figure~3 of \cite{Mielnik1996}. We can see that our truncated eigenvalues are all strictly smaller than their untruncated counterparts given by \eqref{eq:unperturbed_spec}, and this difference vanishes as $M\to \infty$, as is intuitively reasonable.

Thus we can utilize Theorem~\ref{thm:unbdd_case} to derive the following bounds in $\mathcal L^2(\R,\Gamma)$:

\begin{lemma}
    For our Ornstein--Uhlenbeck diffusion \eqref{eq:OU} with quasi-stationary distribution $\pi=\mathcal N(0,1/2)$, for any $M\ge 2$, there exists a perturbed quasi-stationary distribution $\pi_M$, and writing $\varphi_0, \widehat{\varphi}_{0,M}$ for the corresponding normalized (in $\mathcal L^2(\R,\Gamma)$) eigenfunctions, we have
    \begin{equation*}
        \|\varphi_0 -  \widehat{\varphi}_{0,M}\|_2 \le c_2 \exp(-M/2),
    \end{equation*}
    with $c_2=\frac{\sqrt{3}\cdot 35^{1/4}}{2\delta}$. 
    \label{prop:OU_exam_L2}
\end{lemma}
Note that we do not prove  Assumption~\ref{assm:spec_sec_Lambda1}, but rather, as detailed above, we use the numerical method of \cite{Mielnik1996} to verify that indeed for this example $\delta >0$ as required. For computation of its value, see Remark~\ref{rmk:constants_uni}.
\begin{proof}
    From Theorem~\ref{thm:unbdd_case}, we only need to compute $\|H_M E_0\|=\|H_M \varphi_0\|_2$. With $C=\left( \frac{3}{2\pi}\right)^{1/4}$ we obtain
    \begin{align*}
        \|H_M \varphi_0\|^2_2 &= C^2 \int_{-\infty}^\infty (\kappa(x)-M)_+^2 \exp(-x^2) \dif \Gamma(x)\\
        &= 2C^2 \int_{\sqrt M}^\infty (x^2-M)^2 \exp(-3x^2/2)\dif x\\
        &\le C^2 \left (\int_\R x^8 \exp(-x^2) \dif x\right)^{1/2} \left( \int_{\sqrt M}^\infty \exp(-2x^2)\dif x \right)^{1/2}\\
        &\le C^2 (105 \sqrt{\pi}/16)^{1/2} \left( \sqrt{\pi/2} \exp(-2M) \right)^{1/2},
    \end{align*}
    where we have used Cauchy--Schwarz inequality, the fact that for $\mathcal N(0,\sigma^2)$ the eighth moment is $105\sigma^8$, and standard Gaussian tail bounds; see e.g. \cite[Proposition 2.5]{Wainwright2019}. Tidying up the constants, the lemma is proven.
\end{proof}

We now translate our previous $\mathcal L^2(\R,\Gamma)$ estimate into a bound on the corresponding (normalized) quasi-stationary distributions in $\mathcal L^1(\R,\Gamma)$.
\begin{proposition}
    In the same setting as Lemma~\ref{prop:OU_exam_L2} we obtain 
    for any $M> 2\log\left(2(2\pi)^{1/4} Zc_1\right)\approx 5.58$,
   that there exists a perturbed quasi-stationary distribution $\pi_M$, and
    \begin{equation}
        \label{eq: final_est}
            \int_\R |\pi(x)-\pi_M(x)|\dif x\le c_3\exp(-M/2),
    \end{equation}
    for $c_3\approx 5.49$ 
    defined explicitly in \eqref{eq:c'_for_bound}.
    \label{prop:unbdd_L1}
\end{proposition}
\begin{proof}
    We first relate $\int_\R |\pi(x)- \pi_M(x)|\dif x$ to our previous bound. 
    With $\gamma(x):= \exp(-x^2/2)$ let us write $\phi:= \pi/\gamma$, $\widehat\phi_M:=\pi_M/\gamma$ for the corresponding normalised densities with respect to $\Gamma$. Thus we have $\varphi_0 = \phi/Z$ and $\widehat\varphi_{0,M} = \widehat\phi_M/Z_M$, where the normalizing constants are given by $Z := \|\phi\|_2 = (\frac{2}{3\pi})^{1/4}$ and $Z_M := \|\widehat\phi_M\|_2$.
    
    Thus, with $(\int_{\R} \dif \Gamma)^{1/2}=(2\pi)^{1/4}$,  we have from Lemma~\ref{lemma:L1_to_L2} that
    \begin{align*}
        \int_\R |\pi(x)-\pi_M(x)|\dif x \le
        (2\pi)^{1/4} \|\phi-\widehat\phi_M\|_2 .
    \end{align*}
    From Lemma~\ref{prop:OU_exam_L2}, we obtain
    \begin{equation}
        \left \|\phi - \frac{Z\widehat\phi_M}{Z_M}\right \|_2 \le Zc_2 \exp(-M/2).
        \label{eq:phi_phi_M_bd}
    \end{equation}
    By
    Lemma~\ref{lemma:norm_const_bd} this gives
    \begin{equation*}
        |Z_M-Z|< (2\pi)^{1/4}Z^2 c_2 \frac{\exp(-M/2)}{1-(2\pi)^{1/4} Zc_2\exp(-M/2)},
    \end{equation*}
    provided that $M>2\log((2\pi)^{1/4} Zc_2)$.
    By using Lemma~\ref{lemma:phi_2-bound}
    we also have
    \begin{align*}
        \|\phi-\widehat\phi_M\|_2 &\le 
        \exp(-M/2) \left(  Zc_2+ \frac {(2\pi)^{1/4}Z^2 c_2}{1-(2\pi)^{1/4} Zc_2\exp(-M/2)}  \right).
    \end{align*}
    Furthermore, if $M>2\log\left(2(2\pi)^{1/4} Zc_2\right)$, the final term of the former bracket expression is bounded from above by $2(2\pi)^{1/4} Z^2 c_2$. Putting the pieces together, for all $M>2\log\left(2(2\pi)^{1/4} Zc_2\right)$, by applying Proposition~\ref{prop:L1_pi_bound} we have
    \begin{equation*}
        \int_{\R}|\pi-\pi_M|\dif x \le c_3 \exp(-M/2),
    \end{equation*}
    with 
    \begin{equation}        \label{eq:c'_for_bound}
        c_3 := Zc_2(2\pi)^{1/4}\left( 1+ {2Z(2\pi)^{1/4} } \right).
    \end{equation}
\end{proof}

\begin{remark}
 Note that the numerical solutions plotted in Figures~\ref{fig:gamma_E_M2}, \ref{fig:gamma_E_M10}, \ref{fig:gamma_E_M40} indicate that we may choose $\delta = 1.3$ in our setting. By the former proposition this leads to \eqref{eq: final_est} for
 $M>2 \log\left(2(2\pi)^{1/4} Zc_1\right)\approx 5.58$ and $c_3 \approx 5.49 $.
 \label{rmk:constants_uni}
\end{remark}

The above approach can be applied to the more general multivariate Gaussian case at the outset of this subsection. In this more general case, we are unable to apply our extension of \cite{Mielnik1996}, so can no longer obtain explicit constants as in Remark~\ref{rmk:constants_uni}, but can still bound the effect of truncating at level $M$ up to a constant.
\begin{proposition}\label{prop:multivar_trunc}
    In the multivariate setting $d>1$, given some $\delta$-separation between $\lkap$ and $\widehat \lambda_1^M$, we have the bound for some constant $C_d>0$,
    \begin{equation*}
        \int_{\Rd}|\pi(x) - \pi_M(x)| \dif x \le C_d \delta^{-1} \exp(-M/2).
    \end{equation*}
\end{proposition}
\begin{proof}
    The previous lemmas still go through in the multivariate setting, up to a constant factor which depends on the dimension. In place of a Gaussian tail bound as in the proof of Lemma~\ref{prop:OU_exam_L2}, we can use a $\chi^2$ tail bound, as in, say, \cite[Lemma 1]{Laurent2000}.
\end{proof}

\appendix


\section{Deferred proofs}\label{sec:pf_bdd}

\subsection{Proof of Proposition~\ref{prop: discrete_time_perturbation}}
\label{sec:Discrete_time}

	Let $G(y) := 1-\kappa(y)$ and $\widetilde{G}(y) := 1 - \widetilde{\kappa}(y)$ for all $y\in S_0$. Furthermore let $(Y_n)_{n\in\mathbb{N}_0}$ be a Markov chain on $S_0$ with transition kernel $Q$ starting at $x_0$. Then, for arbitrary
	$A \in \mathcal S_0$,
	with $y_0:= x$, we have
	\begin{align*}
		\mathbb{P}_{x_0}(X_n\in A) 
		& = \int_{S_0} \dots \int_{S_0} \mathbf{1}_A(y_n) \prod_{j=0}^{n-1} G(y_j) Q(y_{n-1},{\rm d}y_n) \dots Q(y_1,{\rm d}y_2) Q(y_0,{\rm d} y_1)\\
		 & = \mathbb{E}_{x_0} \left[\mathbf{1}_A(Y_n) \prod_{j=0}^{n-1} G(Y_j) \right],
		\end{align*}
	and by the same arguments $	\mathbb{P}_{x_0}(\widetilde{X}_n\in A) = \mathbb{E}_{x_0} \left[ \mathbf{1}_A(Y_n) \prod_{j=0}^{n-1} \widetilde{G}(Y_j) \right]$. In addition to that
	\allowdisplaybreaks
	\begin{align*}
	    &  \left \vert \mathbb{P}_{x_0}(X_n\in A\mid X_n\in S_0) 
  			- \mathbb{P}_{x_0}(\widetilde{X}_n\in A\mid \widetilde{X}_n\in S_0)\right \vert \\
= 		& \left \vert \frac{\mathbb{P}_{x_0}(X_n\in A)}{\mathbb{P}_{x_0}(X_n\in S_0)} 
			- \frac{\mathbb{P}_{x_0}(\widetilde{X}_n\in A)}{\mathbb{P}_{x_0}(\widetilde{X}_n\in S_0)} \right \vert \\
\leq 	& \left \vert \frac{\mathbb{P}_{x_0}(X_n\in A)}{\mathbb{P}_{x_0}(X_n\in S_0)} 
				- \frac{\mathbb{P}_{x_0}(\widetilde{X}_n\in A)}{\mathbb{P}_{x_0}(X_n\in S_0)} \right \vert 
			+ \left\vert  \frac{\mathbb{P}_{x_0}(\widetilde{X}_n\in A)}{\mathbb{P}_{x_0}(X_n\in S_0)}
				- \frac{\mathbb{P}_{x_0}(\widetilde{X}_n\in A)}{\mathbb{P}_{x_0}(\widetilde{X}_n\in S_0)} \right \vert\\
\leq 	& \frac{\mathbb{E}_{x_0}\left \vert \prod_{j=0}^{n-1} G(Y_j) 
			- \prod_{j=0}^{n-1} \widetilde{G}(Y_j) \right \vert}{\mathbb{P}_{x_0}(X_n\in S_0)} 
			+ \frac{\mathbb{P}_{x_0}(\widetilde{X}_n\in A)\cdot \mathbb{E}_{x_0}\left \vert \prod_{j=0}^{n-1} G(Y_j) 
				- \prod_{j=0}^{n-1} \widetilde{G}(Y_j) \right \vert}{\mathbb{P}_{x_0}(X_n\in S_0) \mathbb{P}_{x_0}(\widetilde{X}_n\in S_0)}\\
\leq 	&  \frac{2 \cdot \mathbb{E}_{x_0}\left \vert \prod_{j=0}^{n-1} G(Y_j) 
				- \prod_{j=0}^{n-1} \widetilde{G}(Y_j) \right \vert}{\mathbb{P}_{x_0}(X_n\in S_0)}.
	\end{align*}  
By the same arguments, just by adding and subtracting $\frac{\mathbb{P}_{x_0}(X_n\in A)}{\mathbb{P}_{x_0}(\widetilde{X}_n\in S_0)}$ instead of $\frac{\mathbb{P}_{x_0}(\widetilde{X}_n\in A)}{\mathbb{P}_{x_0}(X_n\in S_0)}$ before the first inequality of the previous estimation, one obtains
\begin{align*}
		&  \left \vert \mathbb{P}_{x_0}(X_n\in A\mid X_n\in S_0) 
	- \mathbb{P}_{x_0}(\widetilde{X}_n\in A\mid \widetilde{X}_n\in S_0)\right \vert 
	\leq 	
  \frac{2  \mathbb{E}_{x_0}\left \vert \prod_{j=0}^{n-1} G(Y_j) 
		- \prod_{j=0}^{n-1} \widetilde{G}(Y_j) \right \vert}{\mathbb{P}_{x_0}(\widetilde{X}_n\in S_0)},
\end{align*}
such that
\begin{align} \label{eq: 1st_estimate}
& \left \vert \mathbb{P}_{x_0}(X_n\in A\mid X_n\in S_0) 
- \mathbb{P}_{x_0}(\widetilde{X}_n\in A\mid \widetilde{X}_n\in S_0)\right \vert  
\leq  \frac{2  \mathbb{E}_{x_0}\left \vert \prod_{j=0}^{n-1} G(Y_j) 
	- \prod_{j=0}^{n-1} \widetilde{G}(Y_j) \right \vert}{\max\{\mathbb{P}_{x_0}(\widetilde{X}_n\in S_0),\mathbb{P}_{x_0}(X_n\in S_0)\}}.
\end{align}
By induction over $n$ one can verify for any $z_0,\dots,z_{n-1}\in S_0$ that
\[
 \prod_{j=0}^{n-1} G(z_j) 
 - \prod_{j=0}^{n-1} \widetilde{G}(z_j)	= \sum_{i=0}^{n-1} \left[ \prod_{j=0}^{i-1} G(z_j) (G(z_i)-\widetilde G(z_i)) \prod_{m=i+1}^{n-1} \widetilde{G} (z_m) \right].
\] 
Using this we get
\begin{equation} \label{eq: 2nd_estimate}
	\mathbb{E}_{x_0}\left \vert \prod_{j=0}^{n-1} G(Y_j) 
	- \prod_{j=0}^{n-1} \widetilde{G}(Y_j) \right \vert
	\leq  \Vert \kappa-\widetilde{\kappa}\Vert_{\infty} \sum_{i=0}^{n-1} 
	\mathbb{E}_{x_0} \left \vert \prod_{j=0}^{i-1}G(Y_j) \prod_{m=i+1}^{n-1} 
				\widetilde{G}(Y_{m})   \right \vert
\end{equation}
and eventually
\begin{align*}
 	&	\mathbb{E}_{x_0} \left \vert \prod_{j=0}^{i-1}G(Y_j) \prod_{m=i+1}^{n-1} \widetilde{G}(Y_{m}) \right\vert\\
= & 	G(x_0) \underbrace{\int_{S_0} \dots \int_{S_0}}_{i-1 \text{\,int.}} \int_{S_0} 
			 \underbrace{\int_{S_0} \dots \int_{S_0}}_{n-i-1 \text{\,int.}}
			 \widetilde{G}(y_{n-1}) Q(y_{n-2},{\rm d}y_{n-1}) 
			 \,\cdots\,
			 \widetilde{G}(y_{i+1}) Q(y_i,{\rm d}y_{i+1})
			 \\
		 & \qquad \quad \times
			 	Q(y_{i-1},{\rm d} y_{i}) 
			 G(y_{i-1}) Q(y_{i-2},{\rm d}y_{i-1}) \cdot \ldots \cdot G(y_1) Q({x_0},{\rm d}y_1)\\
= & G({x_0}) \underbrace{\int_{S_0} \dots \int_{S_0}}_{i-1 \text{\,int.}} \int_{S_0} \mathbb{P}_{y_i}(\widetilde{X}_{n-i-1}\in S_0) Q(y_{i-1},{\rm d} y_{i}) 
G(y_{i-1}) Q(y_{i-2},{\rm d}y_{i-1}) \cdot \ldots \cdot G(y_1) Q({x_0},{\rm d}y_1)\\	
\leq & \sup_{z\in S_0} \int_{S_0} \mathbb{P}_{\widetilde{x}}(\widetilde{X}_{n-i-1}\in S_0) Q(z,{\rm d} \widetilde{x})  		 
 \cdot \mathbb{P}_{x_0}(X_{i}\in S_0) \\
\leq & \sup_{z\in S_0} \int_{S_0} \widetilde{c}_u(\widetilde{x}) Q(z,{\rm d}\widetilde{x})\, \widetilde{\alpha}^{n-i-1} \cdot c_u({x_0}) \alpha^{i}
 ,
\end{align*}
where the last inequality follows from \eqref{al: exp_decay_discrete_time1} and \eqref{al: exp_decay_discrete_time2}. Additionally, this conditions lead to \[
 \max\left\{\mathbb{P}_{x_0}(\widetilde{X}_n\in S_0),\mathbb{P}_{x_0}(X_n\in S_0)  \right\}	\geq \min\{ c_\ell({x_0}), \widetilde{c}_\ell({x_0}) \} \max\{\widetilde{\alpha},\alpha\}^n.
\]
Hence by the previous two estimates as well as \eqref{eq: 1st_estimate} and \eqref{eq: 2nd_estimate}  we obtain
\begin{align*}
	\left \vert \mathbb{P}_{x_0}(X_n\in A\mid X_n\in S_0) 
	- \mathbb{P}_{x_0}(\widetilde{X}_n\in A\mid \widetilde{X}_n\in S_0)\right \vert 
	\leq K({x_0}) \Vert \kappa-\widetilde{\kappa} \Vert_\infty	\sum_{i=0}^{n-1} \frac{\alpha^{i}\widetilde{\alpha}^{n-i-1}}{\max\{\alpha,\widetilde{\alpha}\}^n}. 
\end{align*}
Estimating the sum independently for the two cases $\alpha<\widetilde{\alpha}$
and $\alpha\geq\widetilde{\alpha}$ leads to
\[
	\sum_{i=0}^{n-1} \frac{\alpha^{i}\widetilde{\alpha}^{n-i-1}}{\max\{\alpha,\widetilde{\alpha}\}^n} \leq \min\{n,\vert \alpha-\widetilde{\alpha} \vert^{-1}\}
\]
and the statement is proven.

\subsection{Proof of Lemma~\ref{lemma:wey}}
\label{app:weyl}
We prove the inequality for $j=1$. The case $j=0$ works entirely analogously. We have
    \begin{equation*}
        \begin{split}
            \widehat \lambda_1 &= \sup_{\mathbb V \in \mathcal V_1} \inf_{f\in \mathbb V^\perp \cap \mathsf S(\H)\cap \Dom(L)} \langle f, \widehat Lf\rangle\\
            &\ge \sup_{\mathbb V \in \mathcal V_1} \left( \inf_{f\in \mathbb V^\perp \cap \mathsf S(\H)\cap \Dom(L)} \langle f,Lf\rangle +  \inf_{f\in \mathbb V^\perp \cap \mathsf S(\H)\cap \Dom(L)} \langle f,Hf\rangle  \right)\\
            &\ge \sup_{\mathbb V \in \mathcal V_1} \left \{  \inf_{f\in \mathbb V^\perp \cap \mathsf S(\H)\cap \Dom(L)} \langle f,Lf\rangle - \|H\| \right \}\\
            &= \lambda_1 - \|H\| .
        \end{split}
    \end{equation*}
    Exchanging the roles of $\lambda_1$ and $\widehat\lambda_1$, we conclude \eqref{eq:eval_bd} for $j=1$. 

\subsection{Proof of Lemma~\ref{lem:conseq_Weyl_Ass2}}
\label{app:lambda0}
By Lemma~\ref{lemma:wey} and Assumption~\ref{assm:H_small}, we have that
	\[
		\widehat{\lambda}_0 < \lambda_0+\frac{\nu}{2}, \qquad
		\lambda_1-\frac{\nu}{2} < \widehat{\lambda}_1.
	\]
Adding the inequalities and exploiting $\nu = \lambda_1-\lambda_0$ gives the first statement.

We now turn to the second statement.
Clearly we have	$\widehat\lambda_0 \in \sigma(\widehat L)$, see \cite[Theorem XIII.1]{Reed1978}. This is an eigenvalue since $\widehat\lambda_0$ is an isolated point in $\sigma(\widehat L)$. 
The fact that this is furthermore a 1-dimensional eigenspace follows from the simplicity of $\lambda_0$: suppose not, so we have two orthogonal eigenfunctions of $\widehat\lambda_0$, $\varphi_1, \varphi_2$ each with norm 1. Then given any $1$-dimensional linear subspace $\mathbb{V}\subseteq \mathcal{H}$, for some $\alpha\in [0,1]$ we have that $\varphi^*:=\alpha \varphi_1 + \sqrt{1-\alpha^2}\varphi_2 \in  \mathbb{V}^\perp\cap \mathsf{S}(\mathcal{H})\cap \mathcal{D}(L)$, and note that $\varphi^*$ is also an eigenfunction of $\widehat\lambda_0$ with norm 1. Then we have that
\begin{equation*}
    \langle \varphi^*,L\varphi^*\rangle \le \widehat\lambda_0 +\|H\|\le \lambda_0 + 2\|H\|<\lambda_1,
\end{equation*}
which contradicts the definition of $\lambda_1$ in \eqref{eq:var_eigenv}.
Thus the function which generates the eigenspace and has norm $1$ is unique, up to the sign. Therefore, changing the sign if necessary, we may assume that the function $\widehat\varphi\in \mathsf{S}(\mathcal{H})$ from the eigenspace satisfies
	$\langle \varphi , \widehat \varphi\rangle \ge 0$
	. 

\subsection{Proof of \eqref{eq: DavisKahan_bnd} from Remark~\ref{rem:bdd_DK}}
\label{sec:proof_Davis_Kahan_bnd}
From Lemma~\ref{lem:conseq_Weyl_Ass2} we have that $\widehat \lambda_0$ is an isolated point in the spectrum, i.e., an eigenvalue of $\widehat L$ with eigenvector $\widehat \varphi$. Thus $\widehat L$ is invariant on $\lspan\{\widehat \varphi\}$ and on the orthogonal complement $\lspan\{\widehat \varphi\}^\perp$.

Then writing $\widehat L_1$ for the action of $\widehat L$ on $(\lspan\{\widehat \varphi\})^\perp$ as in \eqref{eq:decom_hatA}, we have that $\sigma(\widehat L_1)\subset [\widehat \lambda_1,\infty)\subset [\lambda_1 - \|H\|,\infty)\subset [\lambda_0 + \nu/2,\infty)$, by Lemma~\ref{lemma:wey} and Assumption~\ref{assm:H_small}.

Translating the sin~$\theta$ theorem of \cite{DavisKahan1970} (Section~\ref{subsec:projection_notation}) into the present setting, we obtain 
\begin{equation*}
        \frac{\nu}{2}\sqrt{1-\langle \widehat \varphi, \varphi \rangle^2} \le \|H \varphi\|,
    \end{equation*}
since $\sigma(L_0) = \{\lambda_0\}$ and $\sigma(\widehat L_1)\subset [\lambda_0 + \nu/2,\infty)$. Here $L_0$ denotes the action of $L$ on $\lspan \{\varphi\}$.

The proof of \eqref{eq: DavisKahan_bnd} from Remark~\ref{rem:bdd_DK} is eventually completed by following a calculation analogous to those in the proof of Theorem~\ref{thm:bdd_case}.


\acks 
\noindent We would like to thank the following for interesting discussions on topics related to this work: Abraham Ng, Murray Pollock, Sam Power, Lionel Riou-Durand. We would like to thank the Oberwolfach Uncertainty Quantification meeting, March 2019 for facilitating discussions which lead to this work. We would particularly like to thank the anonymous referees whose comments substantially improved the paper, including shortening proofs of Theorem~\ref{thm:bdd_case} and Lemma~\ref{lem:conseq_Weyl_Ass2}.

\fund 
\noindent DR acknowledges support of the
DFG within project 389483880.
AQW was funded by EPSRC grant CoSInES (EP/R034710/1).

\competing 
\noindent There were no competing interests to declare which arose during the preparation or publication process of this article.



\bibliographystyle{APT} 
\bibliography{approx_QS}       

\begin{thebibliography}{10}

\bibitem{Aldous1988}
{\sc Aldous, D., Flannery, B. and Palacios, J.~L.} (1988).
\newblock {Two applications of urn processes: the fringe analysis of search trees and the simulation of quasi-stationary distributions of Markov chains}.
\newblock {\em Probability in the Engineering and Informational Sciences\/} {\bf 2,} 293--307.

\bibitem{Baudel2020}
{\sc Baudel, M., Guyader, A. and Leli{\`{e}}vre, T.} (2023).
\newblock {On the Hill relation and the mean reaction time for metastable processes}.
\newblock {\em Stochastic Processes and their Applications\/} {\bf 155,} 393--436.

\bibitem{Baumgartel1985}
{\sc Baumg{\"{a}}rtel, H.} (1985).
\newblock {\em {Analytic perturbation theory for matrices and operators}} vol.~15.
\newblock Birkh{\"{a}}user, Basel.

\bibitem{Bierkens2019}
{\sc Bierkens, J., Fearnhead, P. and Roberts, G.} (2019).
\newblock {The Zig-Zag process and super-efficient sampling for Bayesian analysis of big data}.
\newblock {\em The Annals of Statistics\/} {\bf 47,} 1288--1320.

\bibitem{Bouchard-Cote2018}
{\sc Bouchard-C{\^{o}}t{\'{e}}, A., Vollmer, S.~J. and Doucet, A.} (2018).
\newblock {The Bouncy Particle Sampler: A Nonreversible Rejection-Free Markov Chain Monte Carlo Method}.
\newblock {\em Journal of the American Statistical Association\/} {\bf 113,} 855--867.

\bibitem{Burdzy2000}
{\sc Burdzy, K., Ho{\l}yst, R. and March, P.} (2000).
\newblock {A Fleming–Viot Particle Representation of the Dirichlet Laplacian}.
\newblock {\em Communications in Mathematical Physics\/} {\bf 214,} 679--703.

\bibitem{Champagnat2017}
{\sc Champagnat, N. and Villemonais, D.} (2023).
\newblock {General criteria for the study of quasi-stationarity}.
\newblock {\em Electron. J. Probab.\/} {\bf 28,} 1--84.

\bibitem{Collet2013}
{\sc Collet, P., Mart{\'{i}}nez, S. and Mart{\'{i}}n, J.~S.} (2013).
\newblock {\em {Quasi-Stationary Distributions: Markov Chains, Diffusions and Dynamical Systems}}.
\newblock Probability and its Applications. Springer-Verlag Berlin Heidelberg.

\bibitem{Davies2007}
{\sc Davies, E.~B.} (2007).
\newblock {\em {Linear Operators and their Spectra}}.
\newblock Cambridge University press.

\bibitem{DavisKahan1970}
{\sc Davis, C. and Kahan, W.~M.} (1970).
\newblock {The Rotation of Eigenvectors by a Perturbation. III}.
\newblock {\em Source: SIAM Journal on Numerical Analysis\/} {\bf 7,} 1--46.

\bibitem{DelMoral2003}
{\sc {Del Moral}, P. and Miclo, L.} (2003).
\newblock {Particle approximations of Lyapunov exponents connected to Schr{\"{o}}dinger operators and Feynman–Kac semigroups}.
\newblock {\em ESAIM: Probability and Statistics\/} {\bf 7,} 171--208.

\bibitem{Durmus2021}
{\sc Durmus, A., Guillin, A. and Monmarch{\'{e}}, P.} (2021).
\newblock {Piecewise deterministic Markov processes and their invariant measures}.
\newblock {\em Ann. Inst. H. Poincar{\'{e}} Probab. Statist.\/} {\bf 57,} 1442--1475.

\bibitem{fuhrmann2021wasserstein}
{\sc Fuhrmann, S., Kupper, M. and Nendel, M.} (2023).
\newblock Wasserstein perturbations of markovian transition semigroups.
\newblock {\em Ann. Inst. H. Poincar{\'{e}} Probab. Statist.\/} {\bf 59,} 904--932.

\bibitem{Hislop1996}
{\sc Hislop, P.~D. and Sigal, I.~M.} (1996).
\newblock {\em {Introduction to Spectral Theory}} vol.~113.
\newblock {Springer New York, NY}.

\bibitem{hosseini2018convergence}
{\sc Hosseini, B. and Johndrow, J.~E.} (2023).
\newblock {Spectral gaps and error estimates for infinite-dimensional Metropolis-Hastings with non-Gaussian priors}.
\newblock {\em Annals of Applied Probability\/} {\bf 33,} 1827--1873.

\bibitem{johndrow2017error}
{\sc Johndrow, J. and Mattingly, J.} (2017).
\newblock Error bounds for approximations of {M}arkov chains used in {B}ayesian sampling.
\newblock {\em arXiv preprint\/}.

\bibitem{kartashov2019strong}
{\sc Kartashov, N.} (1996).
\newblock {\em Strong stable Markov chains}.
\newblock De Gruyter.

\bibitem{Kato1995}
{\sc Kato, T.} (1995).
\newblock {\em {Perturbation Theory for Linear Operators}} second edi~ed.
\newblock Classics in mathematics. Springer-Verlag, Berlin.

\bibitem{Kolb2012}
{\sc Kolb, M. and Steinsaltz, D.} (2012).
\newblock {Quasilimiting behavior for one-dimensional diffusions with killing}.
\newblock {\em Annals of Probability\/} {\bf 40,} 162--212.

\bibitem{Kulkarni2008}
{\sc Kulkarni, S.~H., Nair, M.~T. and Ramesh, G.} (2008).
\newblock {Some properties of unbounded operators with closed range}.
\newblock {\em Proc. Indian Acad. Sci.\/} {\bf 118,} 613--625.

\bibitem{Kumar2019}
{\sc Kumar, D.} (2019).
\newblock {On a Quasi-Stationary Approach to Bayesian Computation, with Application to Tall Data}.
\newblock {\em PhD thesis}.
\newblock University of Warwick.

\bibitem{Laurent2000}
{\sc Laurent, B. and Massart, P.} (2000).
\newblock {Adaptive Estimation of a Quadratic Functional by Model Selection}.
\newblock {\em The Annals of Statistics\/} {\bf 28,} 1302--1338.

\bibitem{Lelievre2015}
{\sc Leli{\`{e}}vre, T.} (2015).
\newblock {Accelerated dynamics: Mathematical foundations and algorithmic improvements}.
\newblock {\em European Physical Journal: Special Topics\/} {\bf 224,} 2429--2444.

\bibitem{Lewis1979}
{\sc Lewis, P. A.~W. and Shedler, G.~S.} (1979).
\newblock {Simulation of nonhomogeneous poisson processes by thinning}.
\newblock {\em Naval Research Logistics Quarterly\/} {\bf 26,} 403--413.

\bibitem{Mailler2020}
{\sc Mailler, C. and Villemonais, D.} (2020).
\newblock {Stochastic approximation on noncompact measure spaces and application to measure-valued p{\'{o}}lya processes}.
\newblock {\em Annals of Applied Probability\/} {\bf 30,} 2393--2438.

\bibitem{Medina-Aguayo2018}
{\sc Medina-Aguayo, F., Rudolf, D. and Schweizer, N.} (2020).
\newblock Perturbation bounds for {M}onte {C}arlo within {M}etropolis via restricted approximations.
\newblock {\em Stochastic processes and their applications\/} {\bf 130,} 2200--2227.

\bibitem{Meleard2012}
{\sc M{\'{e}}l{\'{e}}ard, S. and Villemonais, D.} (2012).
\newblock {Quasi-stationary distributions and population processes}.
\newblock {\em Probability Surveys\/} {\bf 9,} 340--410.

\bibitem{Metafune2002}
{\sc Metafune, G., Pallara, D. and Priola, E.} (2002).
\newblock {Spectrum of Ornstein-Uhlenbeck operators in Lp spaces with respect to invariant measures}.
\newblock {\em Journal of Functional Analysis\/} {\bf 196,} 40--60.

\bibitem{Mielnik1996}
{\sc Mielnik, B. and Reyes, M.~A.} (1996).
\newblock {The classical Schr{\"{o}}dinger equation}.
\newblock {\em Journal of Physics A: Mathematical and General\/} {\bf 29,} 6009--6025.

\bibitem{mitrophanov2003stability}
{\sc Mitrophanov, A.~Y.} (2003).
\newblock Stability and exponential convergence of continuous-time {M}arkov chains.
\newblock {\em Journal of Applied Probability\/} {\bf 40,} 970--979.

\bibitem{mitrophanov_2005}
{\sc Mitrophanov, A.~Y.} (2005).
\newblock {Sensitivity and convergence of uniformly ergodic {M}arkov chains}.
\newblock {\em J. App. Prob.\/} {\bf 42,} 1003–1014.

\bibitem{Negrea2021}
{\sc Negrea, J. and Rosenthal, J.~S.} (2021).
\newblock {Approximations of geometrically ergodic reversible Markov chains}.
\newblock {\em Adv. Appl. Probab.\/} {\bf 53,} 981--1022.

\bibitem{Pollock2020}
{\sc Pollock, M., Fearnhead, P., Johansen, A.~M. and Roberts, G.~O.} (2020).
\newblock {Quasi-stationary Monte Carlo and the ScaLE algorithm}.
\newblock {\em Journal of the Royal Statistical Society: Series B (Statistical Methodology), with discussion\/} {\bf 82,} 1167--1221.

\bibitem{Reed1978}
{\sc Reed, M. and Simon, B.} (1978).
\newblock {\em {Methods of Modern Mathematical Physics IV: Analysis of operators}}.
\newblock Academic Press, New York.

\bibitem{roberts1998convergence}
{\sc Roberts, G.~O., Rosenthal, J.~S. and Schwartz, P.~O.} (1998).
\newblock {Convergence properties of perturbed Markov chains}.
\newblock {\em Journal of Applied Probability\/} {\bf 35,} 1--11.

\bibitem{rudolf2018}
{\sc Rudolf, D. and Schweizer, N.} (2018).
\newblock {Perturbation theory for Markov chains via Wasserstein distance}.
\newblock {\em Bernoulli\/} {\bf 24,} 2610--2639.

\bibitem{Rudolf2020}
{\sc Rudolf, D. and Wang, A.~Q.} (2020).
\newblock {Discussion of ``Quasi-stationary Monte Carlo and the ScaLE algorithm'' by Pollock, Fearnhead, Johansen and Roberts}.
\newblock {\em Journal of the Royal Statistical Society: Series B (Statistical Methodology)\/} {\bf 82,} 1214--1215.

\bibitem{Seneta2006}
{\sc Seneta, E.} (2006).
\newblock {\em {Non-negative Matrices and Markov Chains}} revised~ed.
\newblock Springer Series in Statistics. Springer.

\bibitem{Simon1993}
{\sc Simon, B.} (1993).
\newblock {Large time behaviour of the heat kernel: on a theorem of Chavel and Karp}.
\newblock {\em Proceedings of the American Mathematical Society\/} {\bf 118,} 513--514.

\bibitem{Taylor1989}
{\sc Taylor, J.~C.} (1989).
\newblock {The Minimal Eigenfunctions Characterize the Ornstein-Uhlenbeck Process}.
\newblock {\em The Annals of Probability\/} {\bf 17,} 1055--1062.

\bibitem{Wainwright2019}
{\sc Wainwright, M.~J.} (2019).
\newblock {\em {High-Dimensional Statistics}}.
\newblock Cambridge University Press.

\bibitem{Wang2020}
{\sc Wang, A.~Q.} (2020).
\newblock {Theory of Killing and Regeneration in Continuous-time Monte Carlo Sampling}.
\newblock {\em PhD thesis}.
\newblock University of Oxford.

\bibitem{Wang2019TheoQSMC}
{\sc Wang, A.~Q., Kolb, M., Roberts, G.~O. and Steinsaltz, D.} (2019).
\newblock {Theoretical properties of quasi-stationary Monte Carlo methods}.
\newblock {\em The Annals of Applied Probability\/} {\bf 29,} 434--457.

\bibitem{Wang2020+}
{\sc Wang, A.~Q., Pollock, M., Roberts, G.~O. and Steinsaltz, D.} (2021).
\newblock {Regeneration-enriched Markov processes with application to Monte Carlo}.
\newblock {\em The Annals of Applied Probability\/} {\bf 31,} 703--735.

\bibitem{Wang2020approx}
{\sc Wang, A.~Q., Roberts, G.~O. and Steinsaltz, D.} (2020).
\newblock {An approximation scheme for quasi-stationary distributions of killed diffusions}.
\newblock {\em Stochastic Processes and their Applications\/} {\bf 130,} 3193--3219.

\end{thebibliography}

\end{document}